\documentclass[a4paper,10pt]{article}

\usepackage{amsmath}

\usepackage{geometry}
\usepackage[backref=page]{hyperref}
\usepackage{cleveref}

\usepackage{pifont}
\usepackage{amsfonts, amssymb, bm, cases}

\usepackage{xcolor}
\hypersetup{
  colorlinks,
  linkcolor={red},
  citecolor={blue},
  urlcolor={black}
}

\newcommand{\conv}[1]{\text{co}(#1)}
\renewcommand{\span}[1]{\text{span}(#1)}
\def\endex{\hfill $\blacksquare$}

\newenvironment{proof}{\noindent {\it Proof:}}{\hfill $\Box$ \vspace{0.3cm}}

\def\N{\mathbb{N}}

\def\Q{\mathbb{Q}}
\def\R{\mathbb{R}}
\def\A{\mathbb{A}}
\def\C{\mathbb{C}}

\newtheorem{theorem}{Theorem}
\newtheorem{lemma}{Lemma}
\newtheorem{proposition}{Proposition}
\newtheorem{corollary}{Corollary}

\newtheorem{definition}{Definition}

\newtheorem{remark}{Remark}

\newtheorem{example}{Example}

\usepackage{amssymb}
\usepackage{amsfonts}
\usepackage{mathrsfs} 
\usepackage{graphicx}
\usepackage{color}
\usepackage{ulem}
\usepackage{bm}

\title{\bf On the configurations of four spheres supporting the vertices of a tetrahedron}

\begin{document}

\author{M. Longinetti$^{1}$ \\[.2em] \href{marco.longinetti@unifi.it}{\small (marco.longinetti@unifi.it)}
  \and S. Naldi$^{2,3}$ \\[.2em] \href{simone.naldi@unilim.fr}{\small (simone.naldi@unilim.fr)}}

\footnotetext[1]{Dipartimento DIMAI, Università degli Studi di Firenze, Viale Morgagni 67/a, 50138, Firenze, Italy.}
\footnotetext[2]{Université de Limoges, CNRS, XLIM, UMR 7252, F-87000 Limoges, France.}
\footnotetext[3]{Sorbonne Université, CNRS, LIP6, F-75005 Paris, France.}

\maketitle

\begin{abstract}
  \noindent
  A reformulation of the three circles theorem of Johnson \cite{johnson2} with distance
  coordinates to the vertices of a triangle is explicitly represented in a polynomial system
  and solved by symbolic computation. A similar polynomial system in distance coordinates to
  the vertices of a tetrahedron $T \subset \R^3$ is introduced to represent the configurations
  of four spheres of radius $R^*$, which intersect in one point, each sphere containing three vertices
  of $T$ but not the fourth one. This problem is related to that of computing the largest value
  $R$ for which the set of vertices of $T$ is an $R$-body \cite{MLVRbodies}. For triangular pyramids
  we completely describe the set of geometric configurations with the required
  four balls of radius $R^*$. The solutions obtained by symbolic computation show that triangular
  pyramids are splitted into two different classes: in the first one $R^*$ is unique, in the
  second one three values $R^*$ there exist. The first class can be itself subdivided into two
  subclasses, one of which is related to the family of $R$-bodies.
\end{abstract}


\vspace{0.2cm}
\noindent {\bf Mathematics Subject Classification (MSC2020).}
68W30, 52A30, 51FXX.

\section{Introduction}
\label{sec_intro}

Let $V = \{v_0,v_1,v_2,v_3\} \subset \R^3$ be the set of vertices of a simplex $T=\conv{V}$ with
circumradius $R_T$. Assume that there exist four distinct spheres $S_0,S_1,S_2,S_3 \subset \R^3$ of the
same radius $R^*$ such that
\begin{itemize}
\item[($i$)] $S_j$ contains the vertices $v_i$, for all $i \neq j$,
  \vspace{-0.2cm}
\item[($ii$)] $S_j$ is the boundary of an open ball not containing $v_j$,
  \vspace{-0.2cm}
\item[($iii$)] the intersection of the four spheres is one point $\{O^*\}$.
\end{itemize}
When $T$ is the regular simplex, it is not difficult to prove by a symmetry argument, see also
\cite[Thm 5.6]{MLVRbodies}, that there exist four spheres of radius $R^* = \frac32 R_T$ which
intersect in $O^*$, the center of $T$ (\Cref{fig:introduction}, on the right).
It turns out that $O^*$ does not belong to any open ball of radius greater than
$R^*$ and not containing the set $V$.
When $T \subset \R^n$ is a general simplex, $n\geq 3$, the question whether a configuration of
$n+1$ distinct spheres of radius $R^*>R_T$ satisfying properties ($i$)-($iii$) exists,
is open in its full generality.
In this paper we propose a method based on symbolic computation for solving this problem on
a special class of tetrahedra, in the case $n=3$.

\begin{figure}[!h]
  \begin{center}
    \includegraphics[scale=1]{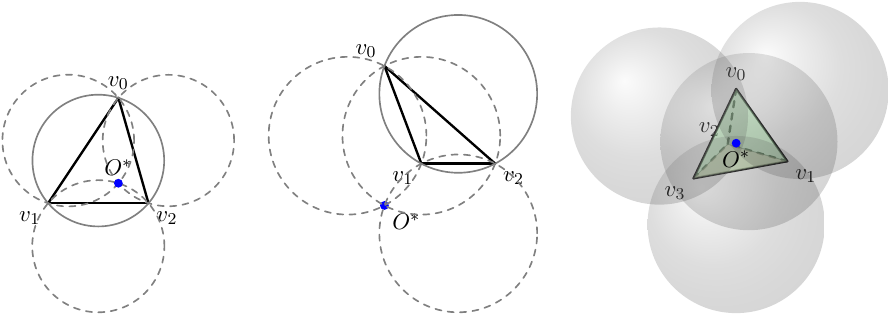}
  \end{center}
  \caption{Configurations of three circles as in \Cref{johnson_theorem}
    for acute-angled (left) and obtuse-angled (center) plane triangles.
    Four spheres satisfying property ($i$) for a simplex in $\R^3$ (right).}
  \label{fig:introduction}
\end{figure}


This geometric configuration is a crucial example in \cite{MLVRbodies} for some sentences and properties
that hold in the class of $R$-bodies in $\R^3$. An {$R$-body} is the complement of the union of
open balls with fixed radius $R$; the smallest $R$-body containing a set $V \subset \R^n$ is called
the {$R$-hulloid} of $V$ and denoted $co_R(V)$.  For a set $V$ of vertices of a simplex $T$ let us call
$R^*_T$ the supremum of all values $r$ such that $V$ is an $r$-body; therefore for $r < R^*_T$ the
complement of $V$ is union of open balls of radius $r$.
The determination of $R^*_T$ for a general simplex $T \subset \R^n$ is closely related to configurations of
$n+1$ spheres satisfying ($i$), ($ii$) and ($iii$).


The planar case ($n=2$), completely solved in \cite[Thm.~4.2]{MLVRbodies}, is somehow special.
Given a triangle $T = co(V) \subset \R^2$, it consists in the problem of determining three circles
of equal radius, each containing two of the three vertices of $T$, but not the third one, and meeting
in one point. In this special case one obtains the equality $R^*_T = R_T$, and that the set $V$ is
a $r$-body if and only if $r \leq R_T$. Indeed, this special property of dimension $2$ follows
directly from a classical result in Euclidean geometry of 1916, namely Johnson's three circles theorem:

\begin{theorem}[Johnson {\cite{johnson2}}]\label{johnson_theorem}
  Assume three distinct plane circles of the same radius $R$ intersect in a point $O^*$ and let $V=\{v_0,v_1,v_2\}$
  be the other intersection points. Then the circle through $V$ has radius $R$ and $O^*$ is the orthocenter of
  $T=\conv{V}$.
\end{theorem}

The four points in $V \cup \{O^*\}$ in \Cref{johnson_theorem} form a rigid configuration in the following sense: every point is the orthocenter of the remaining three. The question whether a generalized version of \Cref{johnson_theorem} holds for $n>2$ is one of the starting points of our work. It is open in its full generality, nevertheless, we give a partial answer for the three-dimensional case in this paper.

Our approach is not based on analytical tools but it is purely algebraic. Indeed the configuration of spheres can be represented by a polynomial system (in some special coordinates), whose coefficients depend on parameters defining the tetrahedron. In particular, we are interested in checking the existence of special configurations of four distinct spheres satisfying properties ($i$) and ($iii$), which have an algebraic nature, and among these, understanding in which case a configuration exists satisfying the addictional property ($ii$), a condition that has a special meaning for the theory of $R$-bodies but that is not algebraic in the parameters of the problem.

\paragraph{Main contributions.}
For a general simplex of $\R^3$ with no special symmetries, the problem of computing the value $R^*_T$ is not straightforward. In \Cref{sec_tetrahedra_general} we present an algebraic formulation for the problem of finding, for a tetrahedron $T \subset \R^3$, the radius $R^*$ of four spheres satisfying conditions ($i$) and ($iii$). This is based on the fact that conditions ($i$) and ($iii$) can be translated into equality constraints (parametric polynomial equations) through the so-called Cayley-Menger determinants, while condition ($ii$) is rather semialgebraic, that is, given by polynomial inequalities.

Whereas the algebraic approach seems challenging for the general class of tetrahedra, we consider in \Cref{sec_tetrahedra} the subclass of triangular pyramids (tetrahedra with a regular base and the fourth vertex equidistant from the base vertices, cf.~\cite{BB}): in this special case the solutions of the polynomial system can be fully characterized. A complete description, depending on the shape of $T$, of all solutions $R^*$ is given in \Cref{prop_piramidal}; their corresponding geometric configurations are described next and the subclass of triangular pyramids with $R^*_T>R_T$ and $O^* \in \text{int}(T)$ is explicitly characterized. We show that for $n=3$, the point $O^*$ defined above is neither unique, nor coincides in general with a remarkable point of the tetrahedron $T$ (except for special cases); in the planar case, by \Cref{johnson_theorem}, $O^*$ is the orthocenter of the triangle, while in $\R^3$ it is neither the orthocenter nor the Monge point of the tetrahedron $T$, except for tetrahedra with special symmetries.

This symbolic approach is considered also in the planar case for a set $V = \{v_0,v_1,v_2\}$ of three non-collinear vertices of a triangle $T = \conv{V}$ with circumradius $R_T$. A polynomial system in distance coordinates is considered in \Cref{sistema polinomial piano} and its algebraic solutions described in \Cref{solutions planar case} and connected with the corresponding geometric configurations in \Cref{planar_geometric}. This new approach carries out a new proof of a more general formulation of Johnson's Theorem introduced prevously in \Cref{double circles}.

Finally in \Cref{sec_rbodies} we consider any simplex $T \subset \R^3$ and make the connection of the main problem with the theory of R-bodies. In this context, the spheres $S_i$ are the boundary of open balls that are said to be $R$-supporting the set $V$.
In this paper we relax this definition saying that four spheres are $R$-supporting $V$ if they satisfy property ($i$). We conclude with a characterization of $R$-hulloids of triangular pyramids.

$R^*$-body configurations for $V$, special systems of spheres $R^*$-supporting $V$ (see \Cref{def:rbodyconf}),
are related to the R-bodies theory. Their existence was also investigated here via computer algebra
methods for triangular pyramids. In case $R^*=R_T$, the existence of $R^*$-body configurations for $V$
is related to the existence of so-called ``unit sphere systems''. These are special arrangements of $n+2$
spheres in $\R^n$, of equal radius, considered in \cite{MT} and \cite{MT2}, see also \Cref{remark_maehara}
in \Cref{pyramid_geometric}.


\paragraph{Paper outline.}
In \Cref{sec_prelim} the basics of computational and distance geometry are introduced, with focus on Cayley-Menger determinants. In \Cref{double circles} a reformulation of the original problem using equations representing double spheres is discussed. The main contributions of the paper are then organized starting from the planar case for a more clear description: \Cref{sec_planetriangles} is completely devoted to the case $n=2$. In \Cref{sec_tetrahedra_general} a polynomial system representing configurations of spheres supporting the vertices of a tetrahedron in $\R^3$ is presented, specialized to the case of regular tetrahedron for which a complete description of the configurations is given. Next a solution for the case $n=3$ for the class of triangular pyramids is given in \Cref{sec_tetrahedra}. In \Cref{sec_rbodies} the connections with the theory of $R$-bodies is presented.

\section{Preliminaries and Notation}
\label{sec_prelim}

\subsection{Basic algebraic geometry}
\label{ssec:alggeometry}
Denote by $K[x_1,\ldots,x_n]$ the ring of polynomials in variables $x_1,\ldots,x_n$ with
coefficients over a field $K$. In this work, we address geometric questions that often depend on
parameters $a_1,\ldots,a_n$ (that are considered free, unless otherwise stated), in which case
$K = \Q(a_1,\ldots,a_k)$; often parameters are specialized to algebraic numbers in which case the
corresponding polynomials have coefficients in some algebraic extension of $\Q$.

We refer to \cite{CLO} for a complete background on computational algebraic geometry and we only
recall the main definitions needed for this work. The complex (resp. real) common vanishing locus
of a family of polynomials $F \subset K[x_1,\ldots,x_n]$ is denoted $\mathscr{Z}(F)$ (resp.
$\mathscr{Z}_{\R}(F)$) and called a ({real}) {algebraic variety}.
Every ideal $I \subset K[x_1,\ldots,x_n]$ is generated by finitely many polynomials $f_1,\ldots,f_s$,
briefly $I=\langle f_1,\ldots,f_s\rangle$ (if $s=1$, $I$ is called principal).
The {radical} of $I$ is the set
$\sqrt{I} = \{f \in K[x_1,\ldots,x_n] : f^m \in I, \exists m \in \N\}$ and of course
$\mathscr{Z}(\sqrt{I}) = \mathscr{Z}(I) = \mathscr{Z}(f_1,\ldots,f_s)$, for
$I=\langle f_1,\ldots,f_s\rangle$. To a set $Z \subset K^n$ one can associate its {vanishing
  ideal} $\mathscr{I}(Z) = \{f \in K[x_1,\ldots,x_n] : f(z)=0, \, \forall z \in Z\}$.
For an algebraically closed field $K$, the operators $\mathscr{Z}$ and $\mathscr{I}$ define a
bijective correspondence between varieties and radical ideals: by Hilbert's Nullstellensatz
\cite[\S 4.1]{CLO}, one has indeed $\mathscr{I}(\mathscr{Z}(I)) = \sqrt{I}$.

A variety $Z$ is called {irreducible} if it is not union of two algebraic varieties properly
contained in $Z$, in which case its vanishing ideal $\mathscr{I}(Z)$ is prime. Every variety is
the finite union of irreducible varieties, called its {irreducible components}.
An ideal $P$ is called {primary} if $fg \in P$ implies either $f \in P$ or $g \in \sqrt{P}$,
in which case one also has that $\sqrt{P}$ is a prime ideal and $\mathscr{Z}(P)$ irreducible.
Every ideal $I$ is decomposable as
finite intersection of primary ideals $I = P_1 \cap P_2 \cap \cdots \cap P_s$, called a {primary
  decomposition} of $I$. Such decomposition in particular yields the irreducible decomposition of the
associated variety: $\sqrt{I} = \sqrt{P_1} \cap \sqrt{P_2} \cap \cdots \cap \sqrt{P_s}$ and
$\mathscr{Z}(P_i)=\mathscr{Z}(\sqrt{P_i})$, $i=1,\ldots,s$ are the irreducible components of $\mathscr{Z}(I)$.
In this work we make use of the software Macaulay2 \cite{M2} for the computation of primary decompositions
and radicals of ideals related to some special polynomial systems.

Finally we often make use of Sturm sequences and of Sturm Theorem, for which we refer to \cite{BPR}.

\subsection{Distance geometry}
\label{ssec:distancegeometry}

For geometric-constrained problems, such as those treated in this work, it is often useful to use {distance
geometry} as a coordinate model instead of the Cartesian coordinates, see for instance \cite{Havel,YZ}. We briefly
recall the main aspects of distance geometry in $\R^n$ and next focus on the case $n=2$ and $n=3$.

Given a set of $n+1$ points $V = \{v_0,\ldots,v_{n}\} \subset \R^n$, denote by $(X_0,\ldots,X_{n})_V$ the vector
of squared Euclidean distances of $v \in \R^n$ to the elements of $V$, meaning that $\lVert v-v_i \rVert_2^2=
X_i$, $i=0,\ldots,n$ and call them the {distance coordinates} of $v$ {with respect to} $V$.
The convex hull of $V$ is $co(V)$.
The following classical rigidity property shows the good definition of distance coordinates to independent sets:

\begin{proposition}\label{lem:rigidity}
  Let $V=\{v_0,\ldots,v_n\} \subset \A^n$ be $n+1$ affinely-independents vectors in the real affine space $\A^n$
  and let $P,Q \in \A^n$ with the same distance coordinates $(X_0,\ldots,X_n)_V$ with respect to $V$. Then $P=Q$.
\end{proposition}

Let now $d_{ij} = \lVert v_i-v_j\rVert_2^2$ be the squared Euclidean dis\-tan\-ce of $v_i$ to $v_j$.
We recall from \cite{Som} that the squared volume of $\conv{V}$ in $\R^n$ is
a constant multiple of the determinant of the {Cayley-Menger matrix}
\begin{equation}
  \text{CM}_V =
  \begin{bmatrix}
    0      & 1      & 1      & 1      & \cdots & 1      \\
    1      & 0      & d_{01} & d_{02} & \cdots & d_{0n} \\
    1      & d_{01} &      0 & d_{12} & \cdots & d_{1n} \\
    1      & d_{02} & d_{12} &      0 &        & d_{2n} \\[-0.4em]
    \vdots & \vdots & \vdots &        & \ddots & \vdots \\
    1      & d_{0n} & d_{1n} & d_{2n} & \cdots &      0 \\
  \end{bmatrix}
\end{equation}
The following proposition, see \cite{Havel} for more details, gives a relation satisfied by the distance coordinates of the point $v$.
\begin{proposition}\label{cayley_menger}
  Let $d_{ij}$, for $0<i<j=1,\ldots,n$, be the squared relative distances of $n+1$ affinely-independent points
  $V=\{v_0,\ldots,v_{n}\}$ and let $D_0,\ldots,D_{n} \geq 0$. Then there exists $v \in \span{V}$
  with distance coordinates $(D_0,\ldots,D_{n})_V$ if and only if the determinant of the augmented
  Cayley-Menger matrix $\mathrm{CM}_{V \cup \{v\}}$ is zero:
  
  \begin{equation}\label{detaugmentedCMv}
    \det
    \begin{bmatrix}
    0    & 1   & 1        & 1        & 1        & \cdots & 1        \\
    1    & 0  & D_0        & D_1      & D_2     &  \cdots & D_n      \\
    1    & D_0   & 0        & d_{01} & d_{02} & \cdots & d_{0n} \\
    1    &  D_1   & d_{01} &      0   & d_{12} & \cdots & d_{1n} \\
    1    &  D_2   & d_{02} & d_{12} &        0 &        & d_{2n} \\[-0.4em]
    \vdots     & \vdots   & \vdots   &   \vdots &      & \ddots &          \\
    1    &  D_n    & d_{0n} & d_{1n} & d_{2n} & \cdots &        0 \\
  \end{bmatrix}=0
\end{equation}
\end{proposition}


Let us consider now the case  $v\in \span{V}$ lies on a sphere of radius $r$, with center $w \in \span V$. 

\begin{proposition}\label{cayley_menger for sphere r}
  Let $d_{ij}$, for $0<i<j=1,\ldots,n$, be the squared relative distances of $n+1$ affinely-independent points
  $V=\{v_0,\ldots,v_{n}\}$ and let $D_0,\ldots,D_{n} \geq 0$ and $\rho=r^2$. Let $w\in \span V$ be the
  center of a sphere $S$ in $\span V$ of radius $r$ containing $V$. Then there exists $v \in \span{V}\cap S$
  with distance coordinates $(D_0,\ldots,D_{n})_V$ if and only if the determinant of the augmented
  Cayley-Menger matrix $\mathrm{CM}_{V \cup \{v\}\cup \{w\}}$ is zero:
   \begin{equation}\label{detaugmentedCMvrho}
  \det
  \begin{bmatrix}
    0 &  1 & 1   & 1        & 1        & 1        & \cdots & 1        \\
    1 &   0   & \rho     & \rho      & \rho  & \rho    &  \cdots  & \rho   \\
    1 &  \rho    & 0  & D_0        & D_1      & D_2     &  \cdots & D_n  \\
    1 &  \rho  & D_0   & 0        & d_{01} & d_{02} & \cdots & d_{0n} \\
    1 &  \rho  &  D_1   & d_{01} &      0   & d_{12} & \cdots & d_{1n} \\
    1 &   \rho &  D_2   & d_{02} & d_{12} &        0 &        & d_{2n} \\[-0.4em]
    \vdots   & \vdots &\vdots  & \vdots   & \vdots   &          & \ddots &          \\
    1  & \rho   &  D_n    & d_{0n} & d_{1n} & d_{2n} & \cdots &        0 \\
  \end{bmatrix}=0
   \end{equation}
\end{proposition}




\subsection{Double circles and spheres}\label{double circles}

Let $V=\{v_0,v_1,v_2\} \subset \Pi$ be a set of three affinely independent points of a real plane $\Pi$ and let
$r \geq \frac12 \lVert v_i-v_j\rVert_2$ for all $0 \leq i < j \leq 2$. Denote by $\mathcal{C}^r_{0},\mathcal{C}^r_{1},
\mathcal{C}^r_{2}$ the union of the two circles through $v_1,v_2$ (resp. $v_0,v_2$ and $v_0,v_1$),
contained in $\Pi$, of radius $r$ and call it the {double circle through $v_1,v_2$} (resp.
{$v_0,v_2$} and {$v_0,v_1$}) {of radius $r$}. For $r=\frac12 \lVert v_i-v_j\rVert_2$, the double circle
$\mathcal{C}^r_k$, $k \neq i,j$, reduces to one circle. Define the intersection of the three double circles
\begin{equation}\label{double_circles}
  \mathcal{C}^r_{V}:=
  \bigcap_{i=0}^2 \mathcal{C}^r_{i}.
\end{equation}
Hereafter we show that Cayley-Menger constraints \eqref{detaugmentedCMv} and \eqref{detaugmentedCMvrho}
naturally define plane double circles and thus can be used to model the configuration in \eqref{double_circles}:

\begin{theorem}\label{prop:double_circles}
  Let $V = \{v_0,v_1,v_2\} \subset \Pi$ be affinely-independent points spanning a plane $\Pi$ and let
  $\mathcal{C}^r_{0},\mathcal{C}^r_{1},\mathcal{C}^r_{2} \subset \Pi$ be the double circles of radius $r$
  as in \eqref{double_circles}.
  Denoted for $k=0,1,2$, $V_k = V \setminus \{v_k\}$ and by $w_k$ any of the circumcenters of $\mathcal{C}^r_k$,
  then $P \in \mathcal{C}^r_V$ if and only if
  \begin{equation}\label{polsyst_plane}
  \det(\mathrm{CM}_{V \cup \{P\}}) = \det(\mathrm{CM}_{V_0 \cup \{w_0,P\}}) =
  \det(\mathrm{CM}_{V_1 \cup \{w_1,P\}}) = \det(\mathrm{CM}_{V_2 \cup \{w_2,P\}}) = 0.
  \end{equation}
\end{theorem}
\begin{proof}
  We prove that $P \in \mathcal{C}^r_{0}$ if and only if
  $\det(\mathrm{CM}_{V \cup \{P\}}) = \det(\mathrm{CM}_{\{w_0,v_1,v_2,P\}}) = 0$: this implies
  the statement.
  First remark that $\text{CM}_{\{w_0,v_1,v_2,P\}}$ is well-defined as it only depends on
  the relative distances of $w_0,v_1,v_2,P$ and the two circumcenters $w_0$ have all distance $r$ from
  $v_1,v_2,P$ for $P \in \mathcal{C}^r_{0}$. Further, $P \in \mathcal{C}^r_0$ if and
  only if $P \in \Pi = \span{V}$ and
  $P$ belongs to a circle of radius $r$ containing $v_1$ and $v_2$. By \Cref{cayley_menger},
  the first condition is equivalent to $\det(\text{CM}_{V \cup \{P\}}) = 0$ and by
  \Cref{cayley_menger for sphere r}, the second condition is equivalent to
  $\det(\text{CM}_{V_0 \cup \{P\}}) = 0$: indeed, it holds if and only if $V_0 \cup \{w_0,P\}=
  \{w_0,v_1,v_2,P\}$ spans a plane and that $w_0$ is the circumcenter of $\conv{\{v_1,v_2,P\}}$.
\end{proof}


Using double circles in place of circles, the following proposition is a reformulation of Johnson's
Theorem (\Cref{johnson_theorem}):

\begin{proposition}\label{prop_johnson}
  Let $V = \{v_0,v_1,v_2\}\subset\Pi$ be not collinear points in a real affine plane $\Pi$, $T = \conv{V}$ the
  triangle with vertices in a set
  $V$ and let $R_T$ be the circumradius of $T$. Then $\mathcal{C}^r_{V}= \emptyset$ if and only if $r \neq R_T$.
  For $r = R_T$, then $\mathcal{C}^r_{V}$ consists of the union of the circumcircle of $T$ and of its orthocenter
  $O^*$.
\end{proposition}

The statement of \Cref{prop_johnson} might appear more advanced than that of \Cref{johnson_theorem}, however
its proof is straightforward from the one of \Cref{johnson_theorem} given in \cite{johnson}.
The advantage of this new formulation is twofold. First, it makes the following fact explicit: in order for
the three circles to intersect in one point, it is necessary that their radius coincide with the circumradius of
$T$. This is a special property of dimension two, indeed, this paper shows that it fails in dimension three.
Secondly, with the help of \Cref{prop:double_circles}, the constraint
$P \in \mathcal{C}^r_V$ translates into the system of polynomial equations \eqref{polsyst_plane} and thus
allows to use computer algebra in order to give an algebraic proof of \Cref{prop_johnson} and at the same time
of the classical Johnson's Theorem \cite{johnson2}, which is what we do in \Cref{sec_planetriangles}.


\vspace{0.5cm}

In a three dimensional real space $E$, a similar formulation can be developed for the set of
vertices of a tetrahedron and a configuration of four spheres, such as those introduced at the beginning of \Cref{sec_intro}.
Let $V = \{v_0,v_1,v_2,v_3\} \subset E$ be affinely independent.
Let us denote by $\mathcal{S}^r_i$,
$i=0,1,2,3$, the union of the two spheres containing $V_i := V \setminus \{v_i\}$, for $r \geq r_i$, where
$r_i$ is the circumradius of $\conv{V_i}$ in dimension two (that is in $\span{V_i}$).
For $r \geq \max_i \{r_i\}$, let
\begin{equation}\label{double_spheres}
  \mathcal{S}^r_V :=
  \bigcap_{i=0}^3 \mathcal{S}^r_i.
\end{equation}

The following is a direct generalization of \Cref{prop:double_circles} to the case $n=3$:

\begin{proposition}\label{prop_johnson_tridimensional}
  Let $V = \{v_0,v_1,v_2,v_3\} \subset E$ be the set of vertices of a
  tetrahedron $T = \conv{V}$ in a three-dimensional
  real affine space $E$, with $\span{V}=E$. Let $w_i$ denote any of the two circumcenters of
  $\mathcal{S}^r_i$, for $i=0,\ldots,3$. Then $P \in \mathcal{S}^r_V$ if and only if
  \begin{equation}\label{polsyst_space}
    \begin{aligned}
      \det(\mathrm{CM}_{V \cup \{P\}}) & = 0 \\
      \det(\mathrm{CM}_{V_i \cup \{w_i,P\}}) & = 0, \text{ for } i = 0, \ldots, 3.
    \end{aligned}
  \end{equation}
\end{proposition}

Let us remark that the distance coordinates of the circumcenters $w_i$ of $\mathcal{C}^r_V$ in
\Cref{prop:double_circles} (and of $\mathcal{S}^r_V$ in \Cref{prop_johnson_tridimensional}) coincide with
the unique value $r$, which is an unknown of the polynomial systems \eqref{polsyst_plane} and \eqref{polsyst_space}.
Therefore the number of variables in such systems equals the number of their equations, indeed, the distances
$\lVert w_i-v_i\rVert_2$, for $i=0,\ldots,3$, are all equal to the unknown $r$.

\section{Plane triangles}
\label{sec_planetriangles}

This section contains a proof of \Cref{prop_johnson} and as a byproduct, an alternative proof of Johnson's Theorem
\cite{johnson2} (\Cref{johnson_theorem}) based on Cayley-Menger determinants and symbolic computation.
The equations in \eqref{polsyst_plane} defining the intersection of three double circles
(cf. \eqref{double_circles}) are polynomials in the distance coordinates, with coefficients depending on
the parameters of the problem.

\subsection{A polynomial system}\label{sistema polinomial piano}

Let $V=\{v_0,v_1,v_2\} \subset \Pi$ be three non-collinear points and let $A$ (resp. $B$ and $C$)
be the square of the distance of $v_1$ to $v_2$ (resp. $v_0$ to $v_2$ and $v_0$ to $v_1$).
Denote by $V_i = V \setminus \{v_i\}, i=0,1,2$ and by $T = \conv{V}$.
Moreover let $X,Y,Z$ be the squared distances with respect to $v_0,v_1,v_2$ of an unknown fourth point
$P \in \Pi$, so that $P=(X,Y,Z)_V$ and let $\rho$ be the square of the unknown radius of each of the
three double circles in \eqref{double_circles}.

We recall from \Cref{prop:double_circles} that the condition $P \in \mathcal{C}^r_V$ is equivalent to
\eqref{double_circles}, that is to the parametric polynomial system:


\begin{equation}\label{syst_gen_case}
  \begin{aligned}[]
    [\det(\mathrm{CM}_{V \cup \{P\}})]   & \hspace{0.5cm} -2A^2X-2ABC+2ABX+2ABY+2ACX+2ACZ-2AX^2+2AXY+ \\
          & \hspace{1cm}  +2AXZ-2AYZ-2B^2Y+2BCY+2BCZ+2BXY-2BXZ- \\
          & \hspace{1cm}  -2BY^2+2BYZ-2C^2Z-2CXY+2CXZ+2CYZ-2CZ^2 = 0 \\
    [\det(\mathrm{CM}_{V_0 \cup \{w_0,P\}})] & \hspace{0.5cm} \rho(2(AY+AZ+YZ)-A^2-Y^2-Z^2)-AYZ = 0 \\ 
    [\det(\mathrm{CM}_{V_1 \cup \{w_1,P\}})] & \hspace{0.5cm} \rho(2(BX+BZ+XZ)-X^2-B^2-Z^2)-BXZ = 0 \\ 
    [\det(\mathrm{CM}_{V_2 \cup \{w_2,P\}})] & \hspace{0.5cm} \rho(2(CX+CY+XY)-X^2-Y^2-C^2)-CXY = 0 \\ 
  \end{aligned}
\end{equation}

System \eqref{syst_gen_case} consists of four equations depending on four variables $X,Y,Z,\rho$ and whose
coefficients are polynomial functions of three parameters $A,B,C$.

\subsection{Description of the algebraic solutions}\label{solutions planar case}

In this section we give a complete description of the solutions to system \eqref{syst_gen_case}, recovering the classical result by Johnson in the form of \Cref{prop_johnson}. The proofs are computer-algebra assisted and all computations are done with the use of Macaulay2 \cite{M2}.

First, as a direct application of \Cref{cayley_menger for sphere r}, we recall the classical
formula (in distance coordinates) for the circumradius of a triangle $T = \conv{V}$.
Denote by
$$
\theta := 2(AB+AC+BC)-(A^2+B^2+  C^2)=16\, area(T)^2.
$$

\begin{proposition}\label{cayley_menger_circ}
  $R_T^2 = \frac{ABC}{\theta}.$
\end{proposition}

Then we show that there is a unique non-zero value of $\rho$, as function of the parameters defining
the triangle, for which \eqref{syst_gen_case} is solvable, as foreseen by Johnson's result (\Cref{johnson_theorem}).
\begin{lemma}
  \label{lem_generalABC_radius}
  Let $X,Y,Z,\rho$ satisfy system \eqref{syst_gen_case}.
  Then either $\rho=0$ or $\theta=0$ or $\rho=\frac{ABC}{\theta}$.
\end{lemma}
\begin{proof}
  Let $I \subset \Q(A,B,C)[X,Y,Z,\rho]$ be the ideal defined by
  polynomials in \eqref{syst_gen_case}. The elimination ideal $I \cap \Q(A,B,C)[\rho]$ is principal
  and defined by the polynomial
  $$
  \rho(ABC+A^2\rho-2AB\rho+B^2\rho-2AC\rho-2BC\rho+C^2\rho)^2
  $$
  thus if $(X,Y,Z,\rho)$ is a solution, then either $\rho=0$ or $\theta\rho={ABC}$.
\end{proof}

Next, we describe the algebraic solutions of system \eqref{syst_gen_case} in the following theorem.

\begin{theorem}
  \label{prop_generalABC}
  Let $\theta=2(AB+AC+BC)-(A^2+B^2+C^2)$.
  The solutions to system \eqref{syst_gen_case} are organized in the following components:
  \begin{enumerate}
  \item $\rho = \frac{ABC}{\theta}$ and $(X,Y,Z)_V =
    \big(\frac{A(B+C-A)^2}{\theta},\frac{B(A+C-B)^2}{\theta},\frac{C(A+B-C)^2}{\theta}\big)$
    \label{genABC_ortho}
  \item $\rho = \frac{ABC}{\theta}$ and $(A,B,C,X,Y,Z)$ belongs to a complex irreducible subvariety of
    $\C^6$ of codimension two and degree six
    \label{genABC_circum}
  \item $\rho=0$ or $A=0$ or $B=0$ or $C=0$ or $\theta=0$
    \label{genABC_rho0}
    \label{genABC_ABC0}
  \end{enumerate}
\end{theorem}
\begin{proof}
  From \Cref{lem_generalABC_radius}, one has either $\rho=0$ or $\theta=0$ (\Cref{genABC_rho0})
  or $\rho = \frac{ABC}{\theta}$.
  Let $I \subset \Q(A,B,C)[X,Y,Z,\rho]$ be the ideal generated by polynomials in \eqref{syst_gen_case}.
  The elimination ideal $J = (I + \langle\rho \theta-ABC\rangle) \cap \Q(B,C)[X,Y,Z]$ is radical of
  codimension $2$ and degree $6$. Its primary decomposition has five prime components $\mathscr{P}_0,
  \mathscr{P}_1,\mathscr{P}_2,\mathscr{P}_3,\mathscr{P}_4$ of codimension $2,3,3,3,3$ and degree
  $6,1,1,1,3$, respectively. The polynomials generating these components are given in \Cref{app:th2}.

  The first component $\mathscr{P}_0$ is the one of \Cref{genABC_circum}.
  Eliminating two of the three variables $X,Y,Z$ from the component
  $\mathscr{P}_1$,
  one gets the univariate linear polynomials
    $\theta X - A(B+C-A)^2$, $\theta Y - B(A+C-B)^2$ and $\theta Z - C(A+B-C)^2$,
  whose zero set is the singleton in \Cref{genABC_ortho} (unless $\theta=0$, \Cref{genABC_ABC0}).
  The three remaining components are linear spaces and satisfy either $A=0$ or $B=0$ or $C=0$
  (\Cref{genABC_ABC0}).
\end{proof}

\subsection{Geometric solutions}\label{planar_geometric}

In this section we assume that $A,B,C$ are the squared edge lengths of a plane triangle $T$
(that is, their square roots satisfy the triangle inequality). We describe the acceptable geometric
configurations among the solutions given in \Cref{prop_generalABC},
allowing to make the expressions in Items \ref{genABC_ortho} and \ref{genABC_circum} more explicit.
We call $O^*$ the point in $\R^2$ with distance coordinates $(X,Y,Z)_V$ from the vertices of $T$.


The proof of \Cref{cor_generalABC_ortho} is by standard analytic geometry.

\begin{corollary}\label{cor_generalABC_ortho}
  Assume that $A,B,C$ are the squared edge lengths of a plane triangle $T=\conv{V}$. Then the solution
  $O^*$ of \Cref{genABC_ortho} in \Cref{prop_generalABC} is the orthocenter of $T$.
\end{corollary}

\begin{corollary}\label{cor_generalABC_circumcircle}
  Assume that $A,B,C$ are the squared edge lengths of a plane triangle $T=\conv{V}$. Then the real
  solutions in \Cref{genABC_circum} of \Cref{prop_generalABC} form the circumcircle of $T$.
\end{corollary}
\begin{proof}
  Let $(x,y)$ be Cartesian coordinates in $\R^2$. Assume by simplicity that $A=1$ (the case with general $A>0$
  is equivalently obtained by scaling).
  Without loss of generality, we put (with $\theta=2(B+C+BC)-(1+B^2+C^2)$)
  \begin{equation*}
    v_0 =
    \begin{bmatrix}                                                                                                        
      {(1-B+C)}/{2} \\
      {\sqrt{\theta}}/{2}
    \end{bmatrix}, 
    \hspace{1cm}
    v_1 =
    \begin{bmatrix}
      0 \\ 0
    \end{bmatrix},
    \hspace{1cm}
    \text{and }
    \hspace{0.5cm}
    v_2 =
    \begin{bmatrix}
      1 \\ 0
    \end{bmatrix}.
  \end{equation*}
  Let $\mathscr{P}_0$ be the component related to solutions of \Cref{genABC_circum} in \Cref{prop_generalABC}
  (cf. \Cref{app:th2}).
  Substituting to $X,Y,Z$, in the defining polynomials of $\mathscr{P}_0$, the squared distances
  of $(x,y)$ to $v_0,v_1,v_2$, respectively, with $L=\sqrt{\theta}$ and adding the constraint $L^2=\theta$,
  one gets a polynomial system $F \subset \Q(B,C,L)[x,y]$ whose real zero set $\mathscr{Z}_{\R}(F)$ is irreducible and
  defined by the single equation
  $$
  \sqrt{\theta}x^2+\sqrt{\theta}y^2+y(1-B-C)-\sqrt{\theta}x = 0,
  $$
  that is, $\mathscr{Z}_{\R}(F)$ is the circle circumscribing $T$.
\end{proof}

%

\subsection{The equilateral case}

We explain the algebraic solutions given in Corollaries \ref{cor_generalABC_ortho} and
\ref{cor_generalABC_circumcircle} in the special case of the equilateral
plane triangle (\Cref{fig:equilatero}). Let $I \subset \Q[X,Y,Z,\rho]$ be the ideal defined by system
\eqref{syst_gen_case} with $A=B=C=1$.

As in the proof of \Cref{prop_generalABC}, eliminating the distance variables yields the univariate polynomial
$\rho(3\rho-1)^2$. The case $\rho=0$ corresponds to the geometrically-meaningless situation when the
 circles have radius zero. The second case corresponds to the circumradius of the equilateral
triangle with side length $1$, which is $\sqrt{1/3}$. This already shows the property mentioned in
\Cref{prop_johnson}: in order for three circles of radius $r$, each containing two vertices of an equilateral
triangle $T$, to intersect, it is necessary that $r = R_T$ (cf. \Cref{fig:equilatero}).

\begin{figure}[!ht]
  \begin{center}
    \includegraphics{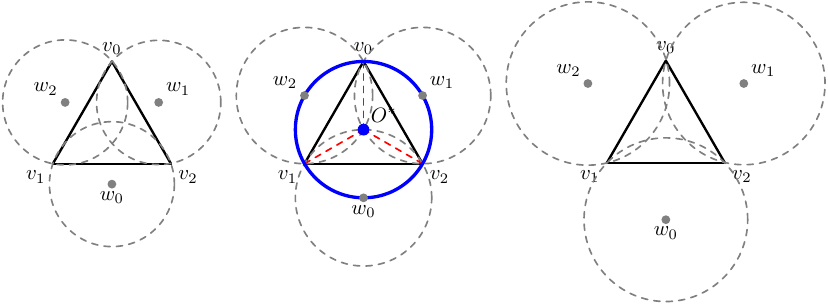}
  \end{center}
  \caption{Johnson's Theorem for a plane equilateral triangle $T$.}
    \label{fig:equilatero}
\end{figure}

One verifies algebraically that this is also sufficient. Indeed, eliminating $\rho$ from $I+\langle 3\rho-1\rangle$,
one gets the ideal $J=I+\langle 3\rho-1\rangle\cap \Q[X,Y,Z]$, which is radical with two prime components:
$\mathscr{P}_1 = \langle X+Y+Z-2, X^2+XY+Y^2-2X-2Y+1\rangle$ and $\mathscr{P}_2 = \langle 3X-1,3Y-1,3Z-1\rangle$.
The component $\mathscr{P}_2$ gives the distance coordinates $(X,Y,Z)_V=(\frac13,\frac13,\frac13)_V$ of the center
of $T$ (the blue dot in \Cref{fig:equilatero}).

Concerning $\mathscr{P}_1$, first we remark that real solutions satisfying $X+Y+Z=2$ correspond to the circumcircle
of $T$. Indeed, let $(x,y)$ be Cartesian coordinates in $\Pi=\span{V}$. Since $\rho=\frac13$ is
the square of the circumradius, we can assume that the three vertices of the triangle are
$v_0=(0,{1}/{\sqrt{3}}), v_1=(-\frac12, -\frac{1}{2\sqrt{3}})$ and $v_2=(\frac12, -\frac{1}{2\sqrt{3}})$.
Thus the condition $X+Y+Z=2$ translates into the condition $3x^2+3y^2+1=2$, that is $x^2+y^2 = \frac13 = \rho$,
the circumcircle. Finally, the second generator yields
$$
X^2+XY+Y^2-2X-2Y+1 = (1/3-x^2-y^2)(\sqrt{3}y-{3}x^2-{3}y^2-{3}x),$$
that is, its zero set is the double circle of radius $\sqrt{1/3}$
through $v_0,v_1$, namely $\mathcal{C}^{\sqrt{1/3}}_2$.
In other words $\mathscr{Z}_\R(F)$, the real algebraic variety defined by system
$$
F =
\begin{cases}
  X+Y+Z = 2 \\
  X^2+XY+Y^2-2X-2Y+1 = 0
\end{cases}
$$
is the circle circumscribing $T$ (\Cref{fig:equilatero}, the blue circle).

\section{Tetrahedra}
\label{sec_tetrahedra_general}

Let $V = \{v_0,v_1,v_2,v_3\} \subset E$ be four affinely-independent points in a three dimensional real
affine space $E$, forming the set of vertices of a tetrahedron $T = \conv{V}$. Let $V_i = V
\setminus \{v_i\}$, $i=0,1,2,3$.

\begin{definition}\label{def_supp_sphere}
  A sphere of radius $r$ containing a set $V_i$ of three vertices of a tetrahedron $T = \conv{V}$
  is called a $r$-supporting sphere of $V$.
\end{definition}

In particular, a collection of spheres $S_i$, $0 \leq i \leq 3$, of radius $r$, satisfying
property ($i$) is a configuration of $r$-supporting spheres of $V$. We are interested in
determining configurations of four distinct $r$-supporting spheres of $V$, intersecting in one
point (property ($iii$)), and among these, in those satisfying the additional property ($ii$).

Let $d_{ij} = \lVert v_i-v_j \rVert_2^2$, for $0 \leq i < j \leq 3$ and let
$P = (X,Y,Z,W)_V$ be a general point.
Assume $P$ belongs to spheres of equal radius $r$, supporting the vertices of a tetrahedron $T$,
but not necessarily satisfying ($ii$), then necessarily $P$ belongs to the set $\mathcal{S}^r_V$ defined in
\eqref{double_spheres} and by \Cref{prop_johnson_tridimensional}, the parameters above satisfy
the determinantal conditions
(denoting $\rho=r^2$ and $|\cdot| := \det(\cdot)$):
\begin{equation}\label{syst:gen:tetra}
  \begin{aligned}
    \scalebox{1.15}{$
      \left|
  \begin{smallmatrix} 
    0 & 1    & 1      & 1      & 1      & 1      \\
    1 & 0    & X      & Y      & Z      & W      \\
    1 & X    & 0      & d_{01} & d_{02} & d_{03} \\
    1 & Y    & d_{01} & 0      & d_{12} & d_{13} \\
    1 & Z    & d_{02} & d_{12} & 0      & d_{23} \\
    1 & W    & d_{03} & d_{13} & d_{23} & 0      \\
  \end{smallmatrix}
  \right|$}
  & = 0 \\
    \scalebox{1.15}{$
    \left|
  \begin{smallmatrix} 
    0 & 1    & 1      & 1      & 1      & 1      \\
    1 & 0    & \rho   & Y      & Z      & W      \\
    1 & \rho & 0      & \rho   & \rho   & \rho   \\
    1 & Y    & \rho   & 0      & d_{12} & d_{13} \\
    1 & Z    & \rho   & d_{12} & 0      & d_{23} \\
    1 & W    & \rho   & d_{13} & d_{23} & 0      \\
  \end{smallmatrix}
  \right| =
  \left|
  \begin{smallmatrix} 
    0 & 1    & 1      & 1      & 1      & 1      \\
    1 & 0    & X      & \rho   & Z      & W      \\
    1 & X    & 0      & \rho   & d_{02} & d_{03} \\
    1 & \rho & \rho   & 0      & \rho   & \rho   \\
    1 & Z    & d_{02} & \rho   & 0      & d_{23} \\
    1 & W    & d_{03} & \rho   & d_{23} & 0      \\
  \end{smallmatrix}
  \right| =
  \left|
  \begin{smallmatrix} 
    0 & 1    & 1      & 1      & 1      & 1      \\
    1 & 0    & X      & Y      & \rho   & W      \\
    1 & X    & 0      & d_{01} & \rho   & d_{03} \\
    1 & Y    & d_{01} & 0      & \rho   & d_{13} \\
    1 & \rho & \rho   & \rho   & 0      & \rho   \\
    1 & W    & d_{03} & d_{13} & \rho   & 0      \\
  \end{smallmatrix}
  \right| =
  \left|
  \begin{smallmatrix} 
    0 & 1    & 1      & 1      & 1      & 1      \\
    1 & 0    & X      & Y      & Z      & \rho   \\
    1 & X    & 0      & d_{01} & d_{02} & \rho   \\
    1 & Y    & d_{01} & 0      & d_{12} & \rho   \\
    1 & Z    & d_{02} & d_{12} & 0      & \rho   \\
    1 & \rho & \rho   & \rho   & \rho   & 0      \\
  \end{smallmatrix}
  \right|$} &= 0
  \end{aligned}
\end{equation}

\begin{definition}\label{def:geom_adm_trivial}
  A real solution $\rho, O^*=(X,Y,Z,W)_V$ to system \eqref{syst:gen:tetra} is called
  \begin{itemize}
  \item geometrically-admissible, if $X,Y,Z,W,\rho > 0$ and there is a configuration
    of four $\rho$-supporting spheres of $V$ meeting in $O^*$;
  \item trivial, if it is geometrically-admissible but $O^*$ lies on the circumscribed sphere of $T = \conv{V}$.
  \end{itemize}
\end{definition}

Giving a complete description of the solutions to this polynomial system with respect to the
distance parameters of the tetrahedron is out of the scope of this paper.
Nevertheless, for the regular tetrahedron it is possible to describe all solutions: this is done
in the next example.

\begin{example}[Regular tetrahedron]\label{ex:regular_non_cyclic}
  There are $7$ non-trivial, geometrically-admissible solutions to system \eqref{syst:gen:tetra}
  for the regular tetrahedron $T \subset \R^3$.
\end{example}
\begin{proof}
  We assume without loss of generality that $d_{01}=d_{02}=d_{03}=1=d_{12}=d_{13}=d_{23}$, so $T = \conv{V}$ is the regular tetrahedron with
  edge $1$. Eliminating $X,Y,Z,W$ from system \eqref{syst:gen:tetra} one gets the equation
  $$
  (8\rho-5)(8\rho-3)^2(32\rho-27)(64\rho^2-8\rho+1)=0.
  $$
  The solution $\rho=\frac38$ corresponds to trivial solutions where $O^*$ lies in the circumsphere
  (indeed in this case $\rho=R^*_T$). The solution $\rho = \frac{27}{32}$ corresponds
  to the case when $O^*$ is the center of $T$ (see also \Cref{ex:regular}).

  For $\rho=\frac58$, each of the remaining variables $X,Y,Z,W$ satisfies the condition $8\chi^2-20\chi+9=0$
  and the symmetric relation $X+Y+Z+W = 5,$ thus there are exactly six solutions $O^*$ for
  $\rho=\frac58$, namely the ones obtained from
  \begin{equation}\label{regular_non_cyclic}
    O^* = (X,Y,Z,W)_V = \Big({\footnotesize \frac{5+\sqrt{7}}{4},\frac{5+\sqrt{7}}{4},\frac{5-\sqrt{7}}{4},\frac{5-\sqrt{7}}{4}}\Big)_V
  \end{equation}
  by applying a permutation of the variables.

  Each of these solutions yields the following configuration of spheres.
  Consider the four double spheres as in \eqref{double_spheres}, of radius
  $r = \sqrt{\rho} = \sqrt{5/8}$ and satisfying property ($i$) for the regular tetrahedron $T$.
  These are boundaries
  of four double open balls of radius $r$: choose two open balls containing the fourth vertex and two open
  balls not containing it (there are $\binom{4}{2}=6$ such choices).
  As an illustrative example (cf. \Cref{fig:tet:reg:noncyclic}), the
  Cartesian coordinates of the point $O^*$ and of the
  centers of the four spheres corresponding to the solution \eqref{regular_non_cyclic} are:
  \begin{equation*}
    \begin{aligned}
      O^* &= \Big(0, -\frac{\sqrt{21}}{6}, \frac{\sqrt{6}-\sqrt{42}}{12}\Big)
      \hspace{0.8cm}
      w_0 = \Big(0,0,-\frac{\sqrt{42}}{12}\Big)
      \hspace{0.8cm}
      w_1 = \Big(0,-\frac{(1+\sqrt{7})\sqrt{3}}{9},\frac{(4+\sqrt{7})\sqrt{6}}{36}\Big) \\
      w_2 &= \Big(\frac{1-\sqrt{7}}{6},\frac{(1-\sqrt{7})\sqrt{3}}{18},\frac{(4-\sqrt{7})\sqrt{6}}{36}\Big) 
      \hspace{0.8cm}
      w_3 = \Big(\frac{\sqrt{7}-1}{6},\frac{(1-\sqrt{7})\sqrt{3}}{18},\frac{(4-\sqrt{7})\sqrt{6}}{36}\Big)
    \end{aligned}
  \end{equation*}
\end{proof}

\begin{figure}[!ht]
  \begin{center}
    \includegraphics[scale=1]{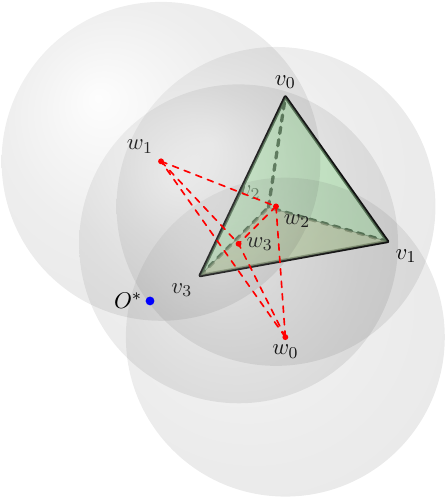}
  \end{center}
  \caption{One of the six solutions with $\rho=\frac58$ for the regular tetrahedron.}\label{fig:tet:reg:noncyclic}
\end{figure}

The solutions in \Cref{ex:regular_non_cyclic} satisfy $R^*>R_T$. For a general tetrahedron
the case $R^*=R_T$ has a connection with results of
Maehara and Tokushige in \cite{MT,MT2} on the existence of some special configurations of
five spheres in $\R^3$.

\begin{remark}\label{remark_maehara}
  Let $T \subset \R^3$ be any tetrahedron, and let $(O^*, \rho)$ be a geometrically-admissible
  solution of system \eqref{syst:gen:tetra}, with $\rho = R_T^2$. Then it is a trivial solution.
\end{remark}
\begin{proof}
  By contradiction, assume that there exist four distinct spheres $S_i, i=0,\ldots,3$, of radius $R_T$,
  satisfying properties
  ($i$) and ($iii$) from the Introduction, whose common intersection point $O^*$ does not lie on
  the circumsphere $S$ of $T$.
  We deduce that the sphere system $\{S,S_i | i=0,\ldots,3\}$ consists of five distinct spheres of equal
  radius, such that each subgroup of four has a common intersection, but the total intersection
  $S \cap \bigcap_{i=0}^3 S_i$ is empty. This is in contradiction with \cite[Thm.8.1]{MT2}.
\end{proof}

\section{Triangular pyramids}
\label{sec_tetrahedra}

In this section we make a twofold simplification to system \eqref{syst:gen:tetra}:
first we restrict it to a special class of tetrahedra (triangular pyramids) and further, we only
look for solutions $(X,Y,Z,W)$ satisfying $Y=Z=W$.

Consider the following geometric configuration: let $V = \{v_0,v_1,v_2,v_3\} \subset E$
be the set of four vertices of a {triangular pyramid} $T = \conv{V}$ in a three-dimensional real affine space $E$,
that is, $T\subset E$ is a tetrahedron such that one of its faces, say $\conv{\{v_1,v_2,v_3\}}$, is
equilateral and the fourth vertex $v_0$ is equidistant to the base vertices.

\begin{figure}[!ht]
  \begin{center}
    \includegraphics[scale=1]{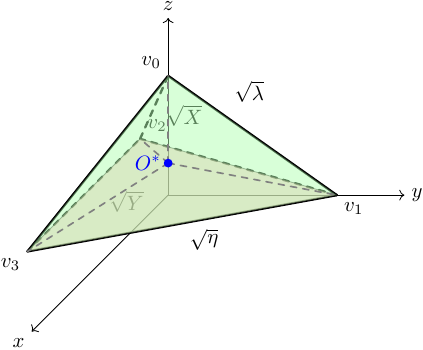}
  \end{center}
  \caption{A triangular pyramid.}
  \label{fig:tetraedro}
\end{figure}

We denote by ${\lambda}$ the squared distance of $v_0$ to each of the base vertices and by ${\eta}$ the squared
edge length of the base face, for some $\lambda,\eta>0$ (cf. \Cref{fig:tetraedro}).
Remark that the two parameters $\lambda,\eta$ identify a solid pyramid if and only if $0 < \eta < 3\lambda$.

We denote by $O^* \in E$ a point satisfying properties ($i$) and ($iii$) given in \Cref{sec_intro}, in other
words, $O^*$ is the unique intersection of four spheres $S_0,S_1,S_2,S_3 \subset \span{V} = \R^3$, of the same radius,
such that for every $0 \leq j \leq 3$, the sphere $S_j$ contains $v_i$, for $i \neq j$
(\Cref{fig:introduction}, right).
In the context of R-bodies (cf. \Cref{sec_rbodies}) we are particularly interested in configurations satisfying
the additional property ($ii$), see \Cref{sec_rbodies} and particularly \Cref{def:rbodyconf}.

Of course, as in the planar case, the trivial configuration where some of the four spheres coincide with
the circumsphere is in principle included in the algebraic solutions of system \eqref{syst:gen:tetra}, but
not interesting for our purposes. By symmetry of $T$, the solution $O^*$ of the polynomial system presented
below in \Cref{ssec_polsyst_tetra} is equidistant to the vertices of the base face, which leads to the trivial
solutions consisting in the north pole and the south pole of the circumsphere. With the aim of generalizing
Johnson's configuration to $\R^3$, we would like to exclude these solutions and this is why we call them
trivial in \Cref{def:geom_adm_trivial}.

\subsection{A polynomial system}
\label{ssec_polsyst_tetra}

Denote by $X$ (resp.~$Y$) the squared distances of $P$ to $v_0$ (resp. $v_1,v_2$ and $v_3$), so that
the vector of distance coordinates of $P$ with respect to $V$ is $(X,Y,Y,Y)_V$. Let $\rho$ be the squared
radius of the four spheres.
The center of the sphere $S_k$ is denoted by $w_k$ and $V_k := V \setminus \{v_k\}$, $k=0,1,2,3$.
System \eqref{syst:gen:tetra} simplifies to the following parametric polynomial system:

\begin{equation}\label{eqnl}
  \begin{aligned}[]
    [\det\text{CM}_{V \cup \{P\}}] & \hspace{1cm} 3(X-Y+\lambda)^2+4\eta X-12\lambda X = 0 \\ 
    [\det\text{CM}_{V_0 \cup \{w_0,P\}}] & \hspace{1cm} 3Y^2-4\rho(3Y-\eta) = 0 \\ 
    [\det\text{CM}_{V_1 \cup \{w_1,P\}}] & \hspace{1cm} 4\rho(4 Y \lambda - (X-Y-\lambda)^2-\eta X) -X(4\lambda Y - \eta X) = 0 \\
  \end{aligned}
\end{equation}
Of course by symmetry of $T$ the constraints $\det\text{CM}_{V_2 \cup \{w_2,P\}}=0$ and
$\det\text{CM}_{V_3 \cup \{w_3,P\}}=0$ are equal to the third constraint in \eqref{eqnl}.
Moreover up to scaling the tetrahedron, one can assume that $\lambda=1$,
which is what we do for most part of the remaining section.

Even if our model uses distance coordinates, the reader may assume that the three vertices $v_1,v_2,v_3$ lie in
the plane $z = 0$ in $\R^3$ and bound an equilateral triangle of size $\sqrt{\eta}$ and the fourth vertex
$v_0$ belongs to the $z$-axis at distance $\lambda=1$ from $v_1,v_2,v_3$, in other words that the Cartesian
coordinates of $V$ are:
\begin{equation}\label{vertices}
  v_0 =
  \left[
    \begin{smallmatrix}
      0 \\
      0 \\
      \sqrt{\frac{3-\eta}{3}}
    \end{smallmatrix}
    \right],
  \hspace{0.5cm}
  v_1 =
  \left[
    \begin{smallmatrix}
      0 \\
      \sqrt{\frac{\eta}{3}}\\
      0
    \end{smallmatrix}
    \right],
  \hspace{0.5cm}
  v_2 =
  \left[
    \begin{smallmatrix}
      -\frac{\sqrt{\eta}}{2} \\
      -\frac12 \sqrt{\frac{\eta}{3}}\\
      0
    \end{smallmatrix}
    \right],
  \hspace{0.5cm}
  \text{and }
  \hspace{0.2cm}
  v_3 =
  \left[
    \begin{smallmatrix}
      \frac{\sqrt{\eta}}{2} \\
      -\frac12 \sqrt{\frac{\eta}{3}}\\
      0
    \end{smallmatrix}
    \right].
\end{equation}
The circumcenter of the tetrahedron with vertices $v_0,v_1,v_2,v_3$ is then
$w = (0,0,\frac{3-2\eta}{2\sqrt{9-3n}})$. By symmetry, $O^*$ is in the $z$-axis, thus
$O^*=(0,0,z)$, for $z$ satisfying
\begin{equation}\label{XYeta}
  X = (z-\sqrt{(3-\eta)/{3}})^2 \text{ and } Y = z^2+{\eta}/{3}.
\end{equation}

\subsection{Description of the algebraic solutions}

First we need a technical lemma.

\begin{lemma}\label{lem:fandg}
  Let
  \begin{equation*}
    \begin{aligned}
      g &= 1024(\eta-3)\rho^3+(-704\eta^2+1920\eta+768)\rho^2+\eta(196\eta^2-732\eta+288)\rho+27\eta^2
      \in \Q(\eta)[\rho] \\
      f &= 432(\eta-3)t^3+108\eta(\eta-2)t^2-9\eta^2(\eta+1)t+\eta^4 \in \Q(\eta)[t]
    \end{aligned}
  \end{equation*}
  and let $\overline{\eta} = \frac{135}{98}+\frac{19}{98}\sqrt{57} \approx 2.841$ be the positive root
  of $49\eta^2-135\eta-12$. The following holds:
  \begin{enumerate}
  \item \label{lem:fandg_it1}
    For $\eta\in(0,\overline{\eta})$, $f$ has exactly one real root and
    for $\eta\in(\overline{\eta},3)$, it has three distinct real roots
  \item \label{lem:fandg_it2}
    For $\eta\in(0,\overline{\eta}) \setminus \{\frac{12}{5}\}$, $g$ has exactly one real root
    and for $\eta\in(\overline{\eta},3) \cup \{\frac{12}{5}\}$,
    it has three real roots counting multiplicities (a double one for $\eta=\frac{12}{5}$ and $\eta=\frac{20}{7}$)
  \item \label{lem:fandg_it3}
    For $\eta=\overline{\eta}$, both $f$ and $g$ have a simple real root and a double real root.
  \end{enumerate}
  All the roots mentioned in \Cref{lem:fandg_it1}, \ref{lem:fandg_it2} and \ref{lem:fandg_it3} are strictly
  positive.
\end{lemma}
\begin{proof}
  The discriminant in $t$ of $f$ is $\Delta_t(f) = 314928\,\eta^7(3-\eta)(49\eta^2-135\eta-12)$ and the discriminant in $\rho$ of $g$ is $\Delta_{\rho}(g) = 196608\,\eta^3(3-\eta)(49\eta^2-135\eta-12)(5\eta-12)^2(7\eta-20)^2$.
  These discriminants vanish if and only if the corresponding polynomial has a double root (cf. \cite[Prop.~4.3]{BPR}) in which case, since they are cubics with real coefficients, they have one real root and one double real root (indeed the roots of a real polynomial come in pairs of conjugate complex numbers).
  This covers the cases $\eta=\frac{12}{5}, \frac{20}{7}$ and \Cref{lem:fandg_it3}. Concerning $f$, one has $\Delta_t(f)<0$ in $(0,\overline{\eta})$ and $\Delta_t(f)>0$ in $(\overline{\eta},3)$, which proves \Cref{lem:fandg_it1}, by applying \cite[Prop.~4.5]{BPR}. The same argument applied to $\Delta_{\rho}(g)$ yields \Cref{lem:fandg_it2}.
  
  Concerning the sign of the roots, for $\eta\in(0,3)$ the leading coefficients of $f$ and $g$ are negative and the constant coefficients are always positive, thus they have at least one positive real root. In particular, when there is a unique real root (in $(0,\overline{\eta}) \setminus \{\frac{12}{5}\}$), it is positive. Now assume $\eta \in (\overline{\eta}, 3) \cup \{\frac{12}{5}\}$: in this interval, the coefficients of polynomials $t \mapsto f(-t)$ and $\rho \mapsto g(-\rho)$ are all positive, thus by Descartes' rule of signs \cite[Thm.~2.33]{BPR}, $f(-t)$ and $g(-\rho)$ do not have positive roots. In other words, the roots of $f$ and $g$ in $(\overline{\eta},3) \cup \{\frac{12}{5}\}$ are positive and non-zero because $f(0)=\eta^4>0$ and $g(0) = 27\eta^2>0$ for $\eta>0$.
\end{proof}

Next we prove our main result.
All the computations in the proof of \Cref{prop_piramidal} are done with Macaulay2 \cite{M2}.

\begin{theorem}
  \label{prop_piramidal}
  Let $\overline{\eta} = \frac{135}{98}+\frac{19}{98}\sqrt{57}$ be the positive root of $49\eta^2-135\eta-12$.
  The solutions to the polynomial system \eqref{eqnl}
  correspond to one of the following cases (unless otherwise stated, $\lambda=1$ and $\eta$ is free):
  \begin{enumerate}
  \item \label{piramidal_northern}
    $(X,Y,\rho) = (0,1,\frac{3}{12-4\eta})$
  \item \label{piramidal_southern}
    $(X,Y,\rho) = (\frac{12}{12-4\eta},\frac{4\eta}{12-4\eta},\frac{3}{12-4\eta})$
  \item \label{piramidal_cubic_onereal}
    for $\eta \in (0,\overline{\eta})$,
    except solutions of \Cref{piramidal_northern} and \Cref{piramidal_southern},
    there is a unique real and positive solution $\rho = \rho(\eta)$ and one corresponding
    point $O^*$
  \item \label{piramidal_cubic_threereal}
    for $\eta \in (\overline{\eta},3)$,
    except solutions of \Cref{piramidal_northern} and \Cref{piramidal_southern},
    there are three real and positive solutions $\rho = \rho_1(\eta), \rho_2(\eta), \rho_3(\eta)$ and
    three corresponding points $O^*_1,O^*_2$ and $O^*_3$
  \item \label{piramidal_eta0} $\lambda=0$ or $\eta \leq 0$ or $\eta \geq 3$
  \item \label{piramidal_complex} complex solutions
  \end{enumerate}
\end{theorem}
\begin{proof}
  First, we exclude the case $\lambda=0$ (\Cref{piramidal_eta0}), substitute $\lambda=1$ in
  \eqref{eqnl}
  and let $I \subset \Q(\eta)[X,Y,\rho]$ be the ideal generated by these polynomials.
  Eliminating $X,Y$ from $I$, one gets a principal ideal
  $I \cap \Q(\eta)[\rho] = \langle\eta(4\eta \rho-12\rho+3)^2 g\rangle$, where
  \begin{equation}\label{polG}
    g = 1024(\eta-3)\rho^3+(-704\eta^2+1920\eta+768)\rho^2+\eta(196\eta^2-732\eta+288)\rho+27\eta^2
  \end{equation}
  Necessarily $\eta(4\eta \rho-12\rho+3)^2 g=0$.
  The case $\eta=0$ is covered by \Cref{piramidal_eta0}.

  Next we consider the case $4\eta \rho-12\rho+3=0$, that is $\rho = \frac{3}{12-4\eta}$.
  Let $J = I+\langle 4\eta \rho-12\rho+3\rangle
  \cap \Q[X,Y,\rho]$. The radical of $J$ has three prime components:
  $$
  \sqrt{J} = \langle Y-1,X\rangle \cap \langle Y-4\rho+1,X-4\rho\rangle \cap
  \langle Y-1,3X+8\rho-6,8\rho^2-6\rho+3\rangle
  $$
  The first component yields the solution $(X,Y,\rho) = (0,1,\frac{3}{12-4\eta})$ of
  \Cref{piramidal_northern}.
  The second one gives $(X,Y,\rho) = (\frac{12}{12-4\eta},\frac{4\eta}{12-4\eta},\frac{3}{12-4\eta})$
  (\Cref{piramidal_southern}).
  Finally, the third component yields complex solutions:
  $(\eta,X,Y,\rho)=((12\rho-3)/4\rho,(6-8\rho)/3,1,\rho)$
  with $\rho=(3\pm i\sqrt{15})/8$ (\Cref{piramidal_complex}).

  Now we turn to the case when $\rho$ and $\eta$ satisfy $g(\eta,\rho)=0$. We add this equation to $I$
  and we eliminate $\eta$ and $\rho$, yielding $K = I+\langle g\rangle \cap \Q[X,Y]$.
  The ideal $K$ is generated by the two polynomials $g_1=(X-Y-1)q$ and $g_2=(Y-4)(Y^2-Y+4)q$ where:
  \begin{equation*}
    q = 3X^3+3X^2Y+XY^2-7Y^3-9X^2-2XY+11Y^2+9X-Y-3.
  \end{equation*}
  The system $g_1=0,g_2=0$ is satisfied in the following cases. Either $(X,Y)=(5,4)$, which gives, by backward
  substitution, the unique solution $(\eta,X,Y,\rho) = (\frac{12}{5},5,4,\frac{5}{4})$ and
  $O^*=(0,0,-{4}/{\sqrt{5}})$, which is a special case of \Cref{piramidal_southern}.
  Or $Y$ satisfies $Y^2-Y+4=0$, leading to complex solutions (\Cref{piramidal_complex}).
  Or $X=Y+1$ and $q=0$, in which case after substitution one gets that $Y$ satisfies
  $Y^5-5Y^4+8Y^3-16Y^2=Y^2(Y-4)(Y^2-Y+4)=0$, which gives $Y=0$ (yielding the solution
  $(\eta,X,Y,\rho)=(0,1,0,\rho)$, \Cref{piramidal_eta0}) or $Y=4$ (again \Cref{piramidal_southern})
  or complex solutions.

  Finally, we are left with the case when the pair $(X,Y)$ satisfies the cubic
  bivariate equation $q=0$. Adding to $I+\langle g,q\rangle$ the polynomials
  $X-(z-s)^2$ and $3Y-z^2-\eta$ and $3s^2+\eta-3$ from \eqref{XYeta}
  and eliminating variables $X,Y,\rho$ and $s$,
  we get an ideal whose radical has prime components:
  $\langle z-1,\eta\rangle \cap \langle z+1,\eta\rangle \cap
  \langle 432(\eta-3)z^6+108\eta(\eta-2)z^4-9\eta^2(\eta+1)z^2+\eta^4\rangle$.
  The first two imply $\eta=0$ (\Cref{piramidal_eta0}) and the last one
  is the following cubic evaluated in $t=z^2$:
  \begin{equation}\label{f_eta}
  f(t) =
  432(\eta-3)t^3+108\eta(\eta-2)t^2-9\eta^2(\eta+1)t+\eta^4 \in \Q(\eta)[t]
  \end{equation}
  By \Cref{lem:fandg}, $f$ and $g$ (resp. in \eqref{f_eta} and in \eqref{polG})
  have exactly one real and positive root for $\eta \in (0,\overline{\eta})$
  and exactly three real and positive roots for $\eta \in (\overline{\eta},3)$.
  Thus for $\eta \in (0,\overline{\eta})$ there is exactly one real positive solution
  for $\rho$ and for this solution at most two solutions for the coordinate $z$ of $O^*$: one
  of these correspond to a positive $s$ ($s = \sqrt{(3-\eta)/3}$) and is acceptable, the other
  one corresponds to the situation where the tetrahedron is reflected with respect to
  the plane $z=0$, which is excluded by our choice in \eqref{vertices}. This is the solution
  in \Cref{piramidal_cubic_onereal}.
  Similarly, when $\eta \in (\overline{\eta}, 3)$, by \Cref{lem:fandg} there are exactly three real
  positive solutions for $\rho$. For each of these solutions, the corresponding value of $z$ can be computed
  by the second equation of system \eqref{eqnl}, combined with \eqref{XYeta}. This gives the solutions
  of \Cref{piramidal_cubic_threereal}.
\end{proof}

\subsection{Geometric solutions}\label{pyramid_geometric}

We first prove that, if the parameter $\eta$ defines a solid tetrahedron, then the solutions of the first two items of \Cref{prop_piramidal} are geometrically-admissible but trivial.

\begin{corollary}\label{sols_trivial}
  For $\eta \in (0,3)$, $R_T^2 = \frac{3}{12-4\eta}$ and
  the solution of \Cref{piramidal_northern} (resp. of \Cref{piramidal_southern})
  of \Cref{prop_piramidal} is the northern (resp. southern) vertex of $T = \conv V$.
\end{corollary}
\begin{proof}
  First we prove that $\frac{3}{12-4\eta} = R_T^2$: this is straightforward
  from \Cref{cayley_menger} with $n=3$, $D_{0}=D_{1}=D_{2}=D_3 = \rho$,
  $d_{01}=d_{02}=d_{03} =1$ and $d_{12}=d_{13}=d_{23} = {\eta}$, indeed, the determinant
  of the matrix in \eqref{detaugmentedCMvrho} is $4\eta^3\rho-12\eta^2\rho+3\eta^2$.
  Of course $(X,Y,Y,Y)_V=(0,1,1,1)_V$ are the distance coordinates of $v_0$. The second solution
  $(X,Y) = (\frac{12}{12-4\eta},\frac{4\eta}{12-4\eta})$, for $\eta \in (0,3)$, corresponds to
  the distance coordinates of a point $v$ whose Cartesian coordinates, applying \eqref{XYeta}, are
  $v=(0,0,-\frac{\eta}{\sqrt{9-3\eta}})$: this point thus satisfies $||v_0-w||^2 = ||w-v||^2 =
  \frac{3}{12-4\eta} = \rho$, where $w$ is the circumcenter of $T$. In other words, $v$ is the
  opposite point of $v_0$ in the circumsphere of $T$.
\end{proof}

The next theorem investigates configurations of spheres, with radius equal to $R_T$
(see also \Cref{remark_maehara}),
satisfying property ($i$) for a triangular pyramid $T$, with only algebraic methods.

\begin{theorem}\label{prop:radius_circumradius}
  Let $T$ be a triangular pyramid with parameter $\eta$.
  A solution $O^*$ to system \eqref{syst:gen:tetra} with $\rho=R_T^2$ is either equidistant
  to the base vertices ($Y=Z=W$) or coplanar with the base vertices (geometrically
  non-admissible) or lies on the sphere circumscribing $T$ (trivial).
\end{theorem}
\begin{proof}
  Recall from the proof of \Cref{sols_trivial}, that $R_T^2=\frac{3}{12-4\eta}$ for $\eta \in (0,3)$.
  Consider system \eqref{syst:gen:tetra} with
  $d_{01}=d_{02}=d_{03}=1$ and $d_{12}=d_{13}=d_{23}=\eta$. Let $I$ be the ideal generated
  by equations in \eqref{syst:gen:tetra}.
  Adding $4\eta\rho-12\rho+3$ to $I$ and eliminating $\eta,\rho$, one gets the ideal
  $(I+\langle 4\eta\rho-12\rho+3 \rangle) \cap \Q[X,Y,Z,W]$ generated by the reducible polynomial
  $$
  (Y^2-YZ+Z^2-YW-ZW+W^2)(3X^2+Y^2-2YZ+Z^2-2YW-2ZW+W^2-6X+3).
  $$
  The zero locus of this polynomial defines the Zariski-closure of the set of solutions $(X,Y,Z,W)$
  to system \eqref{syst:gen:tetra} with $\rho = R_T^2$, for a triangular pyramid.
  Using \eqref{XYeta}, the first factor yields $3 \eta (x^2+y^2)$, with real solutions $x=y=0$
  (the $z$-axis). The second factor yields:
  $$
  -8z\left(\frac12\sqrt{9-3\eta}\left(x^2+y^2+z^2-\frac{\eta}{3}\right)+\left(\eta-\frac32\right)z\right)
  $$
  whose real solutions are union of the plane $z=0$ and the sphere circumscribing $T$. Thus a real
  solution $O^*$ must be either equidistant to base vertices, or trivial or coplanar with
  $v_1,v_2,v_3$: this last case is geometrically non-admissible since there is no sphere containing
  four coplanar vectors in $\R^3$.
\end{proof}

\subsection{Examples}\label{piramidal_examples}

We discuss several examples of tetrahedra showing each of the solutions given in \Cref{prop_piramidal}, Items \ref{piramidal_northern} and \ref{piramidal_southern} (which we call {trivial}) and Items \ref{piramidal_cubic_onereal} and \ref{piramidal_cubic_threereal} ({non-trivial} solutions). Indeed, solutions of Items \ref{piramidal_northern}-\ref{piramidal_southern} are expected as special points of the circumsphere of $T$ (cf. \Cref{sols_trivial}) that correspond to the case when the four spheres coincide with the circumsphere.

We start illustrating \Cref{prop_piramidal} with the case of regular tetrahedron for which there is only one non-trivial solution $O^*$ equidistant to the vertices of the base.

\begin{figure}[!ht]
  \begin{center}
    \includegraphics{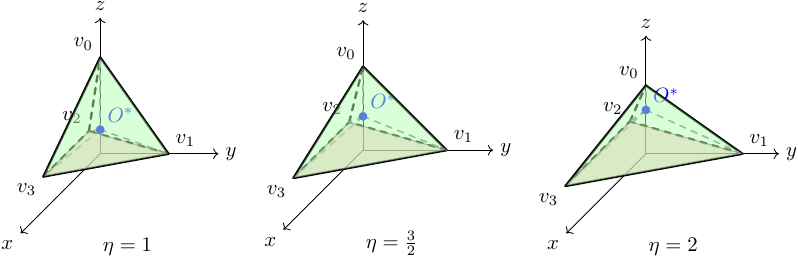}
  \end{center}
  \caption{Regular tetrahedron of \Cref{ex:regular} (on the left); hemispherical pyramid of
    \Cref{ex:hemispherical} (center); trirectangular pyramid of \Cref{ex:trirectangular} (on the right).}
  \label{fig:regular}
\end{figure}

\begin{example}[$\eta=1$, regular tetrahedron]\label{ex:regular}
  Fix $\eta=1$.
  The ideal $I \subset \Q[X,Y,\rho]$ generated by polynomials in \eqref{eqnl} is zero-dimensional of degree $6$.
  Eliminating $X$ and $Y$ from $I$, we get one equation for $\rho$:
  \begin{equation*}
  \begin{aligned}
    131072\rho^5-225280\rho^4  +129536\rho^3-31488\rho^2+3528\rho-243 &= \\
    = (8\rho-3)^2(32\rho-27)(64\rho^2-8\rho+1) &= 0
  \end{aligned}
  \end{equation*}
  cf. also \Cref{ex:regular_non_cyclic}.
  There are two real solutions, $\rho=\frac{3}{8}=R_T^2$ and $\rho=\frac{27}{32}$.
  Replacing $\rho=\frac{3}{8}$, we get the equations
  $2Y^2-3Y+1=X+3Y-3=0$, thus either $(X,Y)=(0,1)$, giving $v_0$, or $(X,Y)=(\frac32,\frac12)$, yielding
  the point $(0,0,-{1}/{\sqrt{6}})$, the trivial solutions of Items
  \ref{piramidal_northern}-\ref{piramidal_southern}.
  The second one $\rho=\frac{27}{32}$, the unique real root of $g$, yields $X=Y=\frac{3}{8}$,
  Cartesian coordinates $O^* = (0,0,{1}/{\sqrt{24}})$, the circumcenter of $T$ (cf. \Cref{fig:regular}, left).
  \endex
\end{example}

The next two examples show again pyramids with a unique real non-trivial solution $O^*$.
The first one is the tetrahedron inscribed in a hemisphere.

\begin{example}[$\eta=\frac32$, hemispherical pyramid]\label{ex:hemispherical}
  For $\eta=\frac32$, the circumcenter is coplanar with $v_1,v_2$, $v_3$ and the pyramid is inscribed in
  a hemisphere. The radius must satisfy $(2\rho-1)^2(2048\rho^3-2752\rho^2+738\rho-81)=0$, thus either
  $\rho=R_T^2=\frac12$ (giving the trivial solutions $(X,Y)=(0,1),(2,1)$) or $\rho\approx 1.0316$, the
  unique real root of the cubic factor. Forcing $2048\rho^3-2752\rho^2+738\rho-81=0$, that is adding this
  polynomial to the ideal, one gets equations $64X^3-120X^2+121X-18 = 64Y^3-88Y^2+45Y-9 =0$, 
  giving $z$ satisfying $16z^3+4\sqrt{2}z^2+2z-\sqrt{2}=0$, that is $z\approx 0.2865$ (\Cref{fig:regular},
  center). \endex
\end{example}

Recall that for a general tetrahedron, the orthocenter
is not well-defined, contrary e.g. to circumcenter, incenter or baricenter. Indeed the four altitudes of a
tetrahedron do not need to be concurrent and the orthocenter is usually replaced for tetrahedra with the
so-called Monge point, see for instance \cite{HHUL}. A tetrahedron where the four altitudes are concurrent
is called orthocentric and this is the case of triangular pyramids, where the orthocenter exists and coincides
with its Monge point.
As the next example shows, $O^*$ differs in general from the orthocenter, and hence from the Monge point,
of a triangular pyramid.

\begin{example}[$\eta=2$, trirectangular pyramid, continued]\label{ex:trirectangular}
  For $\eta=2$, the edges of $T$ containing $v_0$ are
  two-by-two orthogonal: the pyramid is called trirectangular. The Monge point (and orthocenter) of $T$ is $v_0$.
  The unique real root of $g = -1024\rho^3+1792\rho^2-784\rho+108$ is $\rho \approx 1.1746$. The distance
  coordinates $(X,Y,Y,Y)_V$ of $O^*$ must satisfy $36X^3-60X^2+73X-3=0$ and $12Y^3-24Y^2+19Y-6=0$, giving
  $O^* \approx (0,0,0.371)$. In particular $O^*$ is not the Monge point of $T$.
\end{example}

The next three examples are special inasmuch as the parameter $\eta$ is a root of either the discriminant of
$f$ or $g$ (cf. \Cref{lem:fandg}). This represents some bifurcations of the roots of the two polynomials and,
geometrically, to special configurations of spheres.

\begin{figure}[!ht]
  \begin{center}
    \includegraphics{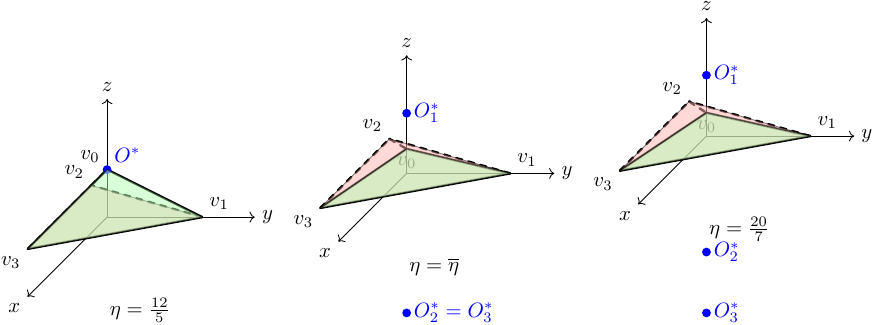}
  \end{center}
  \caption{Tetrahedra of \Cref{ex:12_5} (left), \Cref{ex:discrim} (center) and \Cref{ex:20_7} (right).}
  \label{fig:special}
\end{figure}

\begin{example}[$\eta=\frac{12}{5}$, $\Delta_\rho(g) = 0$ and $\Delta_\rho(f) < 0$]\label{ex:12_5}
  The trivial solutions are $(X,Y) = (0,1)$ and $(5,4)$, with $\rho=R_T^2=\frac54$.
  The polynomial $g$ in \eqref{polG} factors
  $$
  -\frac{3072}{5}\rho^3+\frac{33024}{25}\rho^2-\frac{101952}{125}\rho+\frac{3888}{25}
  = \frac{48}{125}(5-4\rho)(-9+20\rho)^2
  $$
  The first factor gives again $\rho=R_T^2$ and the solutions $O^*=v_0 = (0,1,1,1)_V$.
  However this solution corresponds to a non-trivial configuration of spheres: three spheres of radius
  equal to the circumradius $\sqrt{5/4}$, each one containing $v_0$ and
  two of the base vertices, but not the third one (in other words, the symmetric sphere of the circumsphere
  with respect to a lateral face); and the fourth sphere is the circumsphere. The four
  spheres intersect in $O^*=v_0$.
  The double root of $g$ is $\rho=\frac{9}{20}$ yielding only
  complex solutions: indeed $\Delta_\rho(f)<0$ for $\eta=\frac{12}{5}$ and by elimination $(X,Y)$ must
  satisfy $25X^2-45X+64=25Y^2-45Y+36=0$, that have complex solutions.
  \endex
\end{example}

In the next example ($\eta = \frac{20}{7}$) the polynomial $g$ has again a double root but yielding two
distinct (real) configurations of spheres.

\begin{example}[$\eta=\frac{20}{7}$, $\Delta_\rho(g) = 0$ and $\Delta_\rho(f) > 0$]\label{ex:20_7}
  The circumradius is $\sqrt{21}/2$ and the trivial solutions are $v_0=(0,0,{1}/{\sqrt{21}})$
  (for $X=0,Y=1$) and the south pole $(0,0,-{20}/{\sqrt{21}})$, (for $X=21,Y=20$).
  The polynomial $g$ has a simple root $\rho_2 = \frac{27}{28}$, yielding
  $(X,Y)=(\frac{12}{7},\frac{15}{7})$ and $O^*_2=(0,0,-{5}/{\sqrt{21}})$; and a double root
  $\rho_1=\rho_3=\frac54$, yielding the two solutions $O^*_3=(0,0,\frac{-5\sqrt{21}-21\sqrt{5}}{42})$
  (for $X=\frac{11+\sqrt{105}}{6}, Y=\frac{105+5\sqrt{105}}{42}$) and $O^*_1=(0,0,\frac{-5\sqrt{21}+21\sqrt{5}}{42})$
  (for $X=\frac{11-\sqrt{105}}{6}, Y=\frac{105-5\sqrt{105}}{42}$).
  \endex
\end{example}

The parameter $\eta$ of the following tetrahedron (that we call discriminant tetrahedron) is a zero
of both the discriminants $\Delta_\rho(f)$ and $\Delta_\rho(g)$, thus one of the configurations of
spheres has multiplicity two.

\begin{example}[$\eta=\frac{135}{98}+\frac{19}{98}\sqrt{57}$, discriminant tetrahedron]\label{ex:discrim}
  For $\eta=\overline{\eta}$, $\Delta_\rho(f)=\Delta_\rho(g)=0$ and $f$ and $g$ both have two real
  positive roots, one of which is double. The discriminant tetrahedron is pictured in
  \Cref{fig:special} (center).
  The roots of $g$ are $\rho_1 = \frac{7911+1035\sqrt{57}}{12544} \approx 1.2536$, giving radius
  $\sqrt{\rho_1} \approx 1.1196$ and $\rho_2 = \rho_3 = \frac{9+\sqrt{57}}{16} \approx 1.0344$, giving
  radius $\sqrt{\rho_2} \approx 1.0170$.
  The circumradius is $R_T \approx 2.1739$.
  The non-trivial solutions to system \eqref{eqnl} are given by $O^*_{1} \approx (0,0, 0.5660)$
  and $O^*_{2} = O^*_3 \approx (0,0,-1.3124)$. \endex
\end{example}

%
By \Cref{lem:fandg}, a tetrahedron with $\eta \in (\overline{\eta}, 3)$, admits three distinct configurations of
spheres satisfying properties ($i$) and ($iii$). We end this section with one of such examples.

\begin{example}[$\eta=\frac{29}{10}$]
  For $\eta=\frac{29}{10}$, the squared radius $\rho$ satifies the quintic equation
  $\frac{29}{10}(3-\frac25 \rho)^2 g=0$, with
  $g = -\frac{512}{5}\rho^3+\frac{10384}{25}\rho^2-\frac{135169}{250}\rho+\frac{22707}{100}.$
  The factor $3-\frac25 \rho$ gives the circumradius $\sqrt{\rho}=\sqrt{{15}/{2}} \approx 2.7386$
  and the two trivial solutions.
  The factor $g$ has three distinct real roots, $\rho_1 \approx 1.2370, \rho_2 \approx 0.9687$ and
  $\rho_3 \approx 1.8506$. Adding $g$ to the ideal, we are left with a zero-dimensional ideal of degree
  three, generated by equations
    $14400X^3-109320X^2+129649X-18750 = 0$ and $4800Y^3-45240Y^2+106807Y-73167 = 0$.
  This gives three real solutions in $z$, approximately
  $0.59227$ ($O^*_1$, with radius $\sqrt{\rho_1}$), $-0.93909$ ($O^*_2$, with radius $\sqrt{\rho_2}$)
  and $-2.3005$ ($O^*_3$, with radius $\sqrt{\rho_3}$). \endex
\end{example}


\section{Connections with R-bodies}
\label{sec_rbodies}

This last section aims at making a connection of the results in \Cref{sec_tetrahedra} with the already mentioned
theory of $R$-bodies. All along the section $R$ denotes a fixed positive real number.

A general open ball of radius $r>0$ is denoted by $B(r)$ and for the special case $r=R$
it is simply denoted by $B$.
An $R$-body is the complement of a non empty union of open balls $B$ of radius $R$ in the
Euclidean space $\R^d$, see \cite{MLVRbodies}. These sets have been introduced and studied
by Perkal \cite{Perkal}, used in Walther \cite{Wal} and in Cuevas, Fraiman, Pateiro-L\'opez \cite{Cue}.
For a closed set $E$, its $R$-hulloid $co_R(E)$, namely the minimal $R$-body containing $E$, was
investigated recently in \cite{LMVnew}. In case $E$ is the set $V$ of vertices
of a plane triangle in $\R^2$ or the set of vertices of a tetrahedron $T$ in $\R^3$, a complete
description of $co_r(V)$ is possible for all $r > 0$.

It is straightforward to see, for every simplex $T=\conv{V} \subset \R^n$ given by the convex
hull of a set of $n+1$ affinely independent vectors $V = \{v_0,v_1,\ldots,v_n\} \subset \R^n$, that
for $r \leq R_T$ and for every point $x \in int(T)$, there exists a ball $B(r)$ such that $x \in B(r)$
and $V \cap B(r) = \emptyset$ (such ball is included in the ball $B(R_T)$ circumscribing $T$).
This fact implies that $co_r(V) = V$ for every $r \in (0, R_T]$ and hence the non-trivial case to be
analyzed is $r > R_T$.

A preliminary result in the planar case $n=2$ is in \cite{MLVRbodies}:

\begin{proposition}[{\cite[Thm.~4.2]{MLVRbodies}}]
  \label{cor(v)inplane}
  Let $V = \{v_0,v_1,v_2\} \subset \R^2$ be the set of vertices of a triangle $T = \conv{V}$ with
  circumradius $R_T$.
  If $r > R_T$, then
  $$
  co_r(V) = V \cup \tilde{T}
  $$
  where $\tilde{T} \subset T$ is the curvilinear triangle bordered by three arcs of circles of radius
  $r$, each one through two vertices of $T$. If $T$ is right-angled or obtuse-angled then the vertex
  of the major angle of $T$ is also a vertex of $\tilde{T}$.
\end{proposition}

The result in \cite[Thm.~4.2]{MLVRbodies} is a direct consequence of Johnson's three circles theorem
(\Cref{johnson_theorem} in the Introduction) and implies that $co_r(V)$ has non-empty interior
for every $r > R_T$.
Moreover one can show that for $r \to R_T^+$, the $r$-hulloid $co_r(V)$ tends to $V \cup \{O^*\}$,
where $O^*$ is the orthocenter of $T$, that is, $co_r(V)$ has a discontinuity for $r = R_T$
(indeed as mentioned one always has $co_{R_T}(V) = V$).

Determining the shape of $co_r(V)$ for a tetrahedron $T = \conv{V} \subset \R^3$ is a more subtle problem.

 \begin{definition}\label{defrL(V)} Let us define
   $$
   R^*_T = \inf \{\rho > R_T: co_\rho(V) \text{ has non-empty interior}\}.
   $$ 
  \end{definition}

 Since $co_{R_T}(V)=V$ and $r \mapsto co_r(V)$ is an inclusion-ordered map,
 it turns out that $R^*_T \geq R_T$. Moreover $R^*_T > R_T$ if and only if
 the following property holds:
 \begin{equation}\label{assumptionrhor(V)}
   \exists \, r > R_T: co_r(V)=V.
 \end{equation}
 
A recent result given in \cite{LNV} completes the description of $r$-hulloids of the set of vertices
of a general tetrahedron:

\begin{proposition}[{\cite[Thm.~3.11, Cor.~3.12]{LNV}}]
  \label{finale su T well centered}
  Let $V = \{v_0,v_1,v_2,v_3\}$ be the set of vertices of a tetrahedron
  $T = \conv{V}$ in $\R^3$ with circumradius $R_T$. Under the assumption
  \eqref{assumptionrhor(V)}, then
  there exists $R^* > R_T$ and four spheres $S_j := \partial B_j(R^*), 0 \leq j \leq 3$, such that
  \begin{itemize}
  \item[(i)] $S_j$ contains the vertices $v_i$, for all $i \neq j$;
    \vspace{-0.2cm}
  \item[(ii)] $v_j \not\in B_j(R^*)$;
    \vspace{-0.2cm}
  \item[(iii)] $\bigcap_{j=0}^3 S_j = \{O^*\}$;
    \vspace{-0.2cm}
  \item[(iv)] $O^* \in int(T)$.
  \end{itemize}
  Moreover  
  \begin{eqnarray}\label{rho<rLV} 
   co_r(V) = & V  & \mbox{\quad for \quad }  r < R^* ; \\
  \label{co_R(V)=VcupO*} 
    co_r(V) = & V \cup \{O^*\} & \mbox{\quad for \quad }   r = R^* ; \\  
 \label{co_R(V)=}
 co_r(V)  =& V\cup \tilde{\Gamma} & \mbox{\quad for \quad }   R^* < r,
 \end{eqnarray}
  where  
  \begin{equation}\label{W=}
 \tilde{\Gamma} = int(T) \setminus \bigcup_{j=0}^3 B_j(r)
  \end{equation}
  is a non-empty connected set, and $\partial \tilde{\Gamma}$ is the union of connected subsets of
  $\partial B_j(r), 0 \leq j \leq 3$, the sphere of radius $r$ containing the vertices
  in $V\setminus \{v_j\}$.
\end{proposition}

In \Cref{finale su T well centered} the point $O^*$ and the critical radius $R^*$ are uniquely determined,
and $O^*$ is by construction the circumcenter of the tetrahedron $W$ with vertices the centers of the spheres
$\partial B_j$ and circumradius $R_W = R^*$.
Each ball $B_j(r)$ in the previous proposition is $r$-supporting $V$ at each vertex
in $V\setminus\{v_j\}$, according to {\cite[Def.~2.4]{LNV}} which we recall here:
a ball $B(r)$ is $r$-supporting a closed set $E \subset \R^d$ at $x \in \partial E$
if $x \in \partial B(r)$ and $E \subset (B(r))^c$. Recall that this definition has
been relaxed and extended to spheres in \Cref{def_supp_sphere}.

\begin{definition}\label{def:rbodyconf}
  A sphere configuration $(S_0,S_1,S_2,S_3)$ satisfying properties ($i$), ($ii$), ($iii$)
  and ($iv$) in \Cref{finale su T well centered} is called an $R^*$-body configuration for $V$.
  The spheres $S_j$ are boundaries of open balls $R^*$-supporting $V$.
\end{definition}


Assumption \eqref{assumptionrhor(V)} or equivalently $R_T^*>R_T$,
in \Cref{finale su T well centered}, guarantees the
existence of $R^*$-body configurations for the set of vertices of any tetrahedron.
In case $R^*_T=R_T$, considered previously in \Cref{remark_maehara}, there is a connection with a recent
result \cite[Thm 8.1]{MT2}, showing that there are not $R_T$-body configurations in $\R^3$.
This result is not necessary in case $T$ is a triangular pyramid. Next lemma and theorem give the
complete description when  an $R^*$-body configuration exist, with only algebraic arguments.


\begin{lemma}\label{lem:rbodies_triangular}
  For $0 < \eta < \frac{12}{5}$, then \eqref{assumptionrhor(V)} is satisfied for the triangular
  pyramid $T$ with parameter $\eta$.
\end{lemma}
\begin{proof}
  Put $\rho=r^2$.
  We prove below that for $0 < \eta < \frac{12}{5}$, the polynomial $g$ in \eqref{polG}
  does not vanish in $\rho \in (0,R_T^2]$, where $R_T^2 = \frac{3}{12-4\eta}$ is the squared
    circumradius of $T$. Since the equality $co_r(V) = V$ is always satisfied for
    $r \in (0,R_T]$, we conclude that by continuity there must exist $r > R_T$ with $co_r(V) = V$.

  Consider $g \in \Q(\eta)[\rho]$ as a univariate
  polynomial in $\rho$ with coefficients depending on $\eta$. Its Sturm sequence
  $s = (g_0,g_1,g_2,g_3) \in \Q(\eta)[\rho]^4$ is defined as
  \begin{equation}\label{sturm_g}
    g_0 := g \hspace{0.7cm}
    g_1 := \frac{dg}{d\rho} \hspace{0.7cm}
    g_2 := -\text{rem}_\rho(g_0,g_1) \hspace{0.7cm}
    g_3 := -\text{rem}_\rho(g_1,g_2)
  \end{equation}
  where $\text{rem}_\rho(a,b)$ is the reminder of the division of $a$ by $b$ as polynomials in $\rho$.
  Denote by $s[0] = (v_0,v_1,v_2,v_3) \in \Q(\eta)$ and by $s[\frac{3}{12-4\eta}] = (w_0,w_1,w_2,w_3)
  \in \Q(\eta)$ the evaluation of $s$ at $\rho=0$ and at $\rho=\frac{3}{12-4\eta}$, respectively
  ($s,s[0],s[\frac{3}{12-4\eta}]$ are given explicitly in \Cref{app:sturm}).
  Remark that $v_3 = w_3 = g_3 \in \Q(\eta)$ and that it is always positive in
  $\eta \in (0,\frac{12}{5})$. Moreover $v_0=27\eta^2$ and $w_0$ is a positive multiple of $\eta(12-5\eta)$,
  thus positive in $(0,\frac{12}{5})$.

  Now define $v_1^*,v_2^*$ the unique roots of $v_1,v_2$ in $(0,\frac{12}{5})$: one has $v_1 \geq 0$
  in $(0,v_1^*)$, $v_1 \leq 0$
  in $(v_1^*,\frac{12}{5})$, $v_2\leq 0$ in $(0,v_2^*)$ and $v_2 \geq 0$ in $(v_2^*,\frac{12}{5})$.
  We deduce that for every $\eta \in (0,\frac{12}{5})$, there are exactly two sign variations in $s[0]$.
  Similarly, let $w_1^* < w_1^{**} \in (0,\frac{12}{5})$ be the two real roots of the numerator of $w_1$.
  Then $w_1 \leq 0$ in $(0,w_1^*) \cup (w_1^{**},\frac{12}{5})$ and $w_1 \geq 0$ in $(w_1^*,w_1^{**})$,
  whereas the rational function $w_2$ is always negative in $(0,\frac{12}{5})$. Thus for every
  $\eta \in (0,\frac{12}{5})$, there are exactly two sign variations in $s[\frac{3}{12-4\eta}]$.
  We conclude by Sturm's Theorem \cite[Thm.~2.62]{BPR} that for every $\eta \in (0,\frac{12}{5})$,
  the polynomial $g$ does not vanish in $(0,\frac{3}{12-4\eta})$, as claimed.
\end{proof}



With a similar technique, one can prove the following result.

\begin{lemma}\label{lem:rbodies_triangular2}
  For $\frac{12}{5} \leq \eta < 3$, then every solution $O^*$ to system \eqref{eqnl}
  for the triangular pyramid $T$ with parameter $\eta$,
  is such that $O^* \not\in int(T)$.
\end{lemma}
\begin{proof}
  Recall that according to the Cartesian coordinates in \eqref{vertices}, the north vertex
  of the tetrahedron has coordinates $v_0=(0,0,\sqrt{(3-\eta)/3})$, and the base vertices lie in the
  plane $z=0$.
  Let $O^* = (0,0,z^*)$ in the same Cartesian system.
  Then $t=(z^*)^2$ is one of the roots of $f$ in \eqref{f_eta}.
  For $\eta=\frac{12}{5}$, the unique solution $O^*$ coincides with $v_0$, thus the statement is true.
  We show below that for $\eta \in (\frac{12}{5},3)$, the polynomial $f$ does not vanish for
  $t \in (0,\frac{3-\eta}{3})$, in other words, that $z^* \not\in (0,\sqrt{(3-\eta)/3})$, namely
  $O^* \not\in int(T)$.

  The Sturm sequence $s$ of $f \in \Q(\eta)[t]$ with respect to $t$ is defined as the one in \eqref{sturm_g},
  and is given in \Cref{app:sturm2} together with its evaluations $s[0] = (v_0,v_1,v_2,v_3)$ and 
  $s[\frac{3-\eta}{3}] = (w_0,w_1,w_2,w_3)$ at $t=0$ and
  $t=\frac{3-\eta}{3}$, respectively. The sign pattern of $s[0]$ and $s[\frac{3-\eta}{3}]$ are respectively:
  \begin{equation*}
    s[0]:
    \hspace{0.2cm}
    \begin{cases}
      [+, -, -, +] & \text{ for } \eta \in (\frac{12}{5},v_2^*) \\
      [+,-,+,+] & \text{ for } \eta \in (v_2^*,\overline{\eta}) \\
      [+,-,+,-] & \text{ for } \eta \in (\overline{\eta},3) \\
    \end{cases}
    \hspace{1cm}
    s\left[\frac{3-\eta}{3}\right]:
    \hspace{0.2cm}
    \begin{cases}
      [+,-,-,+] & \text{ for } \eta \in (\frac{12}{5},w_2^*) \\
      [+,-,+,+] & \text{ for } \eta \in (w_2^*,\overline{\eta}) \\
      [+,-,+,-] & \text{ for } \eta \in (\overline{\eta},3) \\
    \end{cases}
  \end{equation*}
  where $v_2^* \approx 2.74$ and $w_2^* \approx 2.71$ are the unique roots of $v_2,w_2$
  in $[\frac{12}{5},3)$. Again applying Sturm's Theorem \cite[Thm.~2.62]{BPR}, we conclude that for
    every $\eta \in (\frac{12}{5}, 3)$, 
  the polynomial $f$ does not have a root in $(0,\frac{3-\eta}{3})$.
\end{proof}

We conclude with the the following result, direct
consequence of \Cref{prop_piramidal} and Lemmas \ref{lem:rbodies_triangular} and \ref{lem:rbodies_triangular2},
summarizing the connection between \Cref{sec_tetrahedra} and the theory of $R$-bodies,
for the class of triangular pyramids.

\begin{theorem}\label{thm_rbodies}
  Let $T = \conv{V}$ be the triangular pyramid with vertices $V$ and parameter $\eta$.
  \begin{enumerate}
  \item For $0 < \eta < \frac{12}{5}$, let $\rho$ be the unique real root of the polynomial $g$
    in \eqref{polG} and let $O^* = (X,Y,Y,Y)_V$ be the unique solution of system \eqref{eqnl}
    related to $\rho$. Let $R^* = \sqrt{\rho}$. Then
    \begin{equation*}
      \begin{aligned}
        & R^* > R_T \\
        & co_{R^*}(V) = V \cup \{O^*\}.
      \end{aligned}
    \end{equation*}
    and $O^* \in int(T)$.
  \item For $\frac{12}{5} \leq \eta \leq 3$, the solutions $(R^*,O^*)$ are geometrically-admissible
    but do not yield $R^*$-body configurations.
  \end{enumerate}
\end{theorem}
\begin{proof}
  By \Cref{lem:rbodies_triangular}, for $\eta \in (0,\frac{12}{5})$, the Assumption in \eqref{assumptionrhor(V)}
  is satisfied. Thus by \Cref{finale su T well centered} there exists an $R^*$-body configuration for $T$
  for some $O^*$ and for some $R^*> R_T$. Remark that $O^*$ must belong to the symmetry axis of the base face,
  thus the unique solution $\rho=(R^*)^2, O^* = (X,Y,Y,Y)$ to system \eqref{eqnl} (that exists and is unique by
  \Cref{prop_piramidal}) satisfies the claim, with $\rho$ the unique real root of $g$. Moreover $O^*$ belongs
  to the segment joining the origin and the apex $v_0$, in particular, $O^* \in int(T)$.

  For $\frac{12}{5} \leq \eta < 3$, by \Cref{lem:rbodies_triangular2}, the solutions $(R^*,O^*)$ given by
  \Cref{prop_piramidal} are such that $O^* \not\in int(T)$, thus the corresponding configuration is not
  $R^*$-body in the sense of \Cref{def:rbodyconf}.
\end{proof}

\paragraph{Acknowledgments.}This work has been partially supported by INDAM-GNAMPA (2023-2024).
The second author is supported by the ANR Project ANR-21-CE48-0006-01 “HYPERSPACE”.

\appendix

\section{Prime decomposition for \Cref{prop_generalABC}}
\label{app:th2}

The generators of $\mathscr{P}_0$ in the proof of \Cref{prop_generalABC} are
\begin{equation*}
\begin{aligned}
 q_1 &= BY^2-BYZ-CYZ+CZ^2-BY-XY-CZ-XZ+2YZ+X \\
 q_2 &= BXY-CXY-BXZ+CXZ+BC-X^2-BY-CZ+YZ+X \\
 q_3 &= BCY-CXY-C^2Z+BXZ-BYZ+CZ^2-BC+CX-XZ+YZ \\
 q_4 &= B^2Y-CXY-BCZ+BXZ-BYZ+CZ^2+BC-BX-BY-CZ-XZ+YZ+X \\
 q_5 &= B^2XZ-2BCXZ+C^2XZ-B^2C+2BCX-CX^2+2BCZ-2BXZ-CZ^2+XZ \\
 q_6 &= CXY^2-2CXYZ+CXZ^2-2CXY-C^2Z-X^2Z+2CYZ+2XYZ-Y^2Z+CX \\
 q_7 &= C^2XY-CX^2Y+BCXZ-2C^2XZ+BX^2Z-CXYZ-BXZ^2+2CXZ^2-BC^2- \\
 &  -BCX+2CX^2+CXY+BCZ+2C^2Z-BXZ-2X^2Z-CYZ+XYZ-2CZ^2+YZ^2+ \\
 & +BC-2CX+2XZ-YZ.
\end{aligned}
\end{equation*}

The further components are
\begin{equation*}
  \begin{aligned}
    \mathscr{P}_1 &= \langle B-C+Y-Z, A-C+X-Z, C^3-2C^2X+CX^2-2C^2Y+2CXY+ \\
    & \hspace{3cm} +CY^2+C^2Z-2CXZ+X^2Z-2CYZ-2XYZ+Y^2Z \rangle \\
    \mathscr{P}_2 &= \langle Z, C, A-B+X-Y\rangle \\
    \mathscr{P}_3 &= \langle Y, B, A-C+X-Z\rangle \\
    \mathscr{P}_4 &= \langle X, A, B-C+Y-Z\rangle. 
  \end{aligned}
\end{equation*}

\section{Sturm sequence for \Cref{lem:rbodies_triangular}}
\label{app:sturm}

The Sturm sequence of the polynomial $g$ in \eqref{polG} is:
\begin{equation*}
  \begin{aligned}
    g_0 &= g \\
    g_1 &= (3072\eta-9216)\rho^2+(-1408\eta^2+3840\eta+1536)\rho+196\eta^3-732\eta^2+288\eta \\
    g_2 &= \frac{(832\eta^4-10560\eta^3+39264\eta^2-43776\eta-4608)\rho+539\eta^5-3483\eta^4+6666\eta^3-2880\eta^2-864\eta}{3(3-\eta)} \\
    g_3 &= -\frac{27(5\eta-12)^2(49\eta^2-135\eta-12)(7\eta-20)^2\eta^3(\eta-3)^2}{(26\eta^4-330\eta^3+1227\eta^2-1368\eta-144)^2}.
  \end{aligned}
\end{equation*}

Its values at $\rho=0$ and $\rho=\frac{3}{12-4\eta}$ are respectively
\begin{equation*}
  \begin{aligned}
  s[0] &=
  \begin{bmatrix}
    v_0 \\
    v_1 \\
    v_2 \\
    v_3
  \end{bmatrix}
  =
  \begin{bmatrix}
    27\eta^2, \\
    4\eta(49\eta^2-183\eta+72), \\
    \frac{\eta(539\eta^4-3483\eta^3+6666\eta^2-2880\eta-864)}{36(3-n)}, \\
    g_3
  \end{bmatrix}
  \\
  s\left[\frac{3}{12-4\eta}\right] &=
  \begin{bmatrix}
    w_0 \\
    w_1 \\
    w_2 \\
    w_3
  \end{bmatrix}
  =
  \begin{bmatrix}
    \frac{12\eta(12-5\eta)(2\eta^2-9\eta+12)}{(\eta-3)^2}, \\
    \frac{4(49\eta^4-330\eta^3+885\eta^2-936\eta+144)}{\eta-3}, \\
    -\frac{539\eta^6-5100\eta^5+16491\eta^4-14958\eta^3-21672\eta^2+35424\eta+3456}{36(\eta-3)^2} \\
    g_3
  \end{bmatrix}
  \end{aligned}
\end{equation*}

\section{Sturm sequence for \Cref{lem:rbodies_triangular2}}
\label{app:sturm2}

The Sturm sequence of the polynomial $f$ in \eqref{f_eta} is:
\begin{equation*}
  \begin{aligned}
    f_0 &= f \\
    f_1 &= (1296\eta-3888)t^2+(216\eta^2-432\eta)t-9\eta^2(\eta+1) \\
    f_2 &= \frac{-\eta^2(-48\eta^2+144\eta-24)t-\eta^2(5\eta^3-13\eta^2-2\eta)}{4\eta-12} \\
    f_3 &= \frac{-9(\eta-3)^2(49\eta^2-135\eta-12)\eta^3}{4(2\eta^2-6\eta+1)^2}.
  \end{aligned}
\end{equation*}

Its values at $t=0$ and $t=\frac{3-\eta}{3}$ are respectively
\begin{equation*}
  \begin{aligned}
  s[0] &=
  \begin{bmatrix}
    v_0 \\
    v_1 \\
    v_2 \\
    v_3
  \end{bmatrix}
  =
  \begin{bmatrix}
    \eta^4 \\
    -9\eta^2(\eta+1) \\
    \frac{\eta^3(5\eta^2-13\eta-2)}{4(3-\eta)} \\
    f_3
  \end{bmatrix}
  \\
  s\left[\frac{3-\eta}{3}\right] &=
  \begin{bmatrix}
    w_0 \\
    w_1 \\
    w_2 \\
    w_3
  \end{bmatrix}
  =
  \begin{bmatrix}
    9(5\eta-12)(2\eta^2-9\eta+12) \\
    63\eta^3-945\eta^2+3456\eta-3888 \\
    \frac{\eta^2(21\eta^3-109\eta^2+150\eta-24)}{4(3-\eta)} \\
    f_3
  \end{bmatrix}
  \end{aligned}
\end{equation*}

\end{document}